\documentclass{compositio}
\usepackage{xypic,mathrsfs,tikz,MnSymbol}
\renewcommand{\mod}{\operatorname{mod}}
\newcommand{\Sub}{\operatorname{Sub}}
\newcommand{\Tor}{\operatorname{Tor}}
\newcommand{\sort}{\operatorname{sort}}

\newcommand{\Inv}{\operatorname{Inv}}

\newcommand{\C}{\mathscr{C}}

\newcommand{\widet}[1]{#1'}
\newcommand{\ak}{{\langle k\rangle}}
\renewcommand{\le}{\reflectbox{{\rm{L}}}}
\newcommand {\F}{\mathbb F_q}
\newcommand{\oF}{\overline{\mathbb F}_q}

\newcommand{\wt}{\widetilde}
\newcommand{\wc}{\widehat {\mathcal{C}}}
\DeclareMathOperator{\Ker}{Ker}
\DeclareMathOperator{\Image}{Im} \renewcommand{\Im}{\Image}
\DeclareMathOperator{\End}{End}
\DeclareMathOperator{\Hom}{Hom}

\DeclareMathOperator{\Ext}{Ext}
\DeclareMathOperator{\modules}{mod} \renewcommand{\mod}{\modules}
\DeclareMathOperator{\Mod}{Mod}

\DeclareMathOperator{\Fac}{Fac}
\newcommand{\Facpp}{\Fac_{\mathrm{pp}}}
\DeclareMathOperator{\Ann}{Ann}

\usetikzlibrary{arrows}

\tikzset{>=stealth'}

\def\arrowLengthDisplayStyle{4ex}
\def\arrowHeightDisplayStyle{.8ex}
\def\arrowSkipDisplayStyle{.5ex}
\def\arrowLengthTextStyle{3ex}
\def\arrowHeightTextStyle{.8ex}
\def\arrowSkipTextStyle{.4ex}
\def\arrowLengthScriptStyle{2.5ex}
\def\arrowHeightScriptStyle{.6ex}
\def\arrowSkipScriptStyle{.3ex}
\def\arrowLengthScriptScriptStyle{2ex}
\def\arrowHeightScriptScriptStyle{.4ex}
\def\arrowSkipScriptScriptStyle{.2ex}

\renewcommand{\to}{\arrow{->}}
\newcommand{\mono}{\arrow{>->}}
\newcommand{\epi}{\arrow{->>}}
\newcommand{\embed}{\arrow{right hook->}}
\renewcommand{\mapsto}{\arrow{|->}}

\newcommand{\MakeTikzArrowWithSuperscriptSubscript}[4]
 {
  \mathchoice
   { %Display style
    \hspace*{\arrowSkipDisplayStyle}
    \begin{tikzpicture}[baseline]
     \draw [#1] (0,\arrowHeightDisplayStyle) -- node [above] {$#2$} node [below] {$#3$} (#4 * \arrowLengthDisplayStyle, \arrowHeightDisplayStyle);
    \end{tikzpicture}
    \hspace*{\arrowSkipDisplayStyle}
   }
   { %Text style
    \hspace*{\arrowSkipTextStyle}
    \begin{tikzpicture}[baseline]
     \draw [#1] (0,\arrowHeightTextStyle) -- node [above] {$\scriptstyle #2$} node [below] {$\scriptstyle #3$} (#4 * \arrowLengthTextStyle, \arrowHeightTextStyle);
    \end{tikzpicture}
    \hspace*{\arrowSkipTextStyle}
   }
   { %Script style
    \hspace*{\arrowSkipScriptStyle}
    \begin{tikzpicture}[baseline]
     \draw [#1] (0,\arrowHeightScriptStyle) -- node [above] {$\scriptscriptstyle #2$} node [below] {$\scriptscriptstyle #3$} (#4 * \arrowLengthScriptStyle, \arrowHeightScriptStyle);
    \end{tikzpicture}
    \hspace*{\arrowSkipScriptStyle}
   }
   { %Script script style
    \hspace*{\arrowSkipScriptScriptStyle}
    \begin{tikzpicture}[baseline]
     \draw [#1] (0,\arrowHeightScriptScriptStyle) -- node [above] {$\scriptscriptstyle #2$} node [below] {$\scriptscriptstyle #3$} (#4 * \arrowLengthScriptScriptStyle, \arrowHeightScriptScriptStyle);
    \end{tikzpicture}
    \hspace*{\arrowSkipScriptScriptStyle}
   }
 }

\newcommand{\MakeTikzArrowWithCentralLabel}[3]
 {
  \mathchoice
   { %Display style
    \hspace*{\arrowSkipDisplayStyle}
    \begin{tikzpicture}[baseline]
     \draw [#1] (0,\arrowHeightDisplayStyle) -- node [fill=white,inner sep=1pt] {$#2$} (#3 * \arrowLengthDisplayStyle, \arrowHeightDisplayStyle);
    \end{tikzpicture}
    \hspace*{\arrowSkipDisplayStyle}
   }
   { %Text style
    \hspace*{\arrowSkipTextStyle}
    \begin{tikzpicture}[baseline]
     \draw [#1] (0,\arrowHeightTextStyle) -- node [fill=white,inner sep=1pt] {$\scriptstyle #2$} (#3 * \arrowLengthTextStyle, \arrowHeightTextStyle);
    \end{tikzpicture}
    \hspace*{\arrowSkipTextStyle}
   }
   { %Script style
    \hspace*{\arrowSkipScriptStyle}
    \begin{tikzpicture}[baseline]
     \draw [#1] (0,\arrowHeightScriptStyle) -- node [fill=white,inner sep=1pt] {$\scriptscriptstyle #2$} (#3 * \arrowLengthScriptStyle, \arrowHeightScriptStyle);
    \end{tikzpicture}
    \hspace*{\arrowSkipScriptStyle}
   }
   { %Script script style
    \hspace*{\arrowSkipScriptScriptStyle}
    \begin{tikzpicture}[baseline]
     \draw [#1] (0,\arrowHeightScriptScriptStyle) -- node [fill=white,inner sep=1pt] {$\scriptscriptstyle #2$} (#3 * \arrowLengthScriptScriptStyle, \arrowHeightScriptScriptStyle);
    \end{tikzpicture}
    \hspace*{\arrowSkipScriptScriptStyle}
   }
 }

\def\arrow#1{\def\lastArrowStyle{#1}
             \futurelet\testchar\arrowMaybeStreched}
\def\arrowMaybeStreched{\ifx[\testchar \let\next\arrowStreched
                         \else \let\next\arrowUnstreched \fi
                        \next}

\def\arrowStreched[#1]{\def\lastArrowStrech{#1}
                       \futurelet\testchar\arrowMaybeLabel}
\def\arrowUnstreched{\def\lastArrowStrech{1}
                     \futurelet\testchar\arrowMaybeLabel}

\def\arrowMaybeLabel{\ifx^\testchar \let\next\arrowSuperscript
                      \else \ifx_\testchar \let\next\arrowSubscript
                             \else \ifx~\testchar \let\next\arrowCentralLabel
                                    \else \let\next\arrowNoLabel
                                   \fi
                            \fi
                     \fi
                     \next}

\def\arrowSuperscript^#1{\def\lastArrowSuperscript{#1}
                         \futurelet\testchar\arrowSuperMaybeSub}
\def\arrowSuperMaybeSub{\ifx_\testchar \let\next\arrowSuperscriptSubscript
                         \else \let\next\arrowSuperscriptNoSubscript \fi
                        \next}

\def\arrowSubscript_#1{\def\lastArrowSubscript{#1}
                         \futurelet\testchar\arrowSubMaybeSuper}
\def\arrowSubMaybeSuper{\ifx^\testchar \let\next\arrowSubscriptSuperscript
                         \else \let\next\arrowSubscriptNoSuperscript \fi
                        \next}

\def\arrowSuperscriptSubscript_#1{\def\lastArrowSubscript{#1}
                                  \arrowDrawSupSub}
\def\arrowSuperscriptNoSubscript{\def\lastArrowSubscript{}
                                 \arrowDrawSupSub}
\def\arrowSubscriptSuperscript^#1{\def\lastArrowSuperscript{#1}
                                  \arrowDrawSupSub}
\def\arrowSubscriptNoSuperscript{\def\lastArrowSuperscript{}
                                 \arrowDrawSupSub}
\def\arrowNoLabel{\def\lastArrowSuperscript{}
                  \def\lastArrowSubscript{}
                  \arrowDrawSupSub}

\def\arrowCentralLabel~#1{\MakeTikzArrowWithCentralLabel{\lastArrowStyle}{#1}{\lastArrowStrech}}
\def\arrowDrawSupSub{\MakeTikzArrowWithSuperscriptSubscript{\lastArrowStyle}{\lastArrowSuperscript}{\lastArrowSubscript}{\lastArrowStrech}}

\newtheorem{theorem}{Theorem}[section]
\newtheorem{corollary}[theorem]{Corollary}

\newtheorem{lemma}[theorem]{Lemma}
\newtheorem{proposition}[theorem]{Proposition}
\newtheorem{conjecture}[theorem]{Conjecture}

\theoremstyle{definition}

\newtheorem{example}[theorem]{Example}
\newcommand{\R}{\mathbf{R}\!}
\renewcommand{\L}{\mathbf{L}\!}
\newcommand{\Db}{{\rm D^b}}
\newcommand{\filt}{{\rm filt}}
\newcommand{\op}{{\rm\scriptstyle op}}
\newcommand{\uperp}{{}^{\upVdash}\!}
\newcommand{\add}{\operatorname{add}}
\begin{document}
\title{Quotient closed subcategories of quiver representations}
\author{Steffen Oppermann}
\address{Department of Mathematics, NTNU, Trondheim, Norway}
\author{Idun Reiten}
\address{Department of Mathematics, NTNU, Trondheim, Norway}
\author{Hugh Thomas}
\address{Department of Mathematics and Statistics, UNB, P.O. Box 4400, Fredericton, Canada}

\numberwithin{equation}{section}

\begin{abstract}
Let $Q$ be a finite quiver without oriented cycles, and let $k$ be an
algebraically closed field.  The main result in this paper is that there
is a natural bijection between the elements in the associated Weyl
group $W_Q$ and the cofinite additive quotient-closed subcategories of the
category of finite dimensional right modules over $kQ$.  We prove this
correpondence by linking these subcategories to certain ideals in the
preprojective algebra associated to $Q$, which are also indexed by elements
of $W_Q$.
\end{abstract}
\maketitle
\section{Introduction}
Let $Q$ be a finite quiver without oriented cycles. 
Let $k$ be an algebraically closed field. The main result in this paper is that
there is a natural bijection between the elements in the associated 
Weyl group $W_Q$ and the cofinite additive quotient closed subcategories of the category $\mod kQ$ of finite dimensional right modules over the path algebra
$kQ$.  Here a subcategory $\mathcal A$ in $\mod kQ$ is called cofinite if there are only a finite number of indecomposable modules of $\mod kQ$ which are not in 
$\mathcal A$.  From now on, when we refer to subcategories, we mean full,
additive subcategories.  

The natural bijection is given via the following map.  Let $\mathcal A$ be a 
cofinite quotient closed subcategory of $\mod kQ$.  We label the
vertices of $Q$ by $1,\dots, n$ so that if $P_1,\dots,P_n$ are the 
corresponding projective modules, then $\Hom(P_i,P_j)=0$ for $i>j$.  List
the indecomposable modules not in $\mathcal A$, starting with the projective
ones, with indices in increasing order, then similarly for $\tau^{-1}P_1,\dots,
\tau^{-1}P_n$, etc., where $\tau$ denotes the AR-translation.  The sequence
of modules gives rise to a word $w$ by replacing $\tau^{-i}P_j$ by the
simple reflection $s_j$ of $W_Q$.  For example, if $Q$ is the quiver 
$1\to 2 \to 3$, and $\tau^{-1}P_1,\tau^{-2}P_1$ are the
indecomposable modules of a quotient closed subcategory of $\mod kQ$,
then the missing indecomposables in the required order are 
$\{P_1,P_2,P_3,\tau^{-1}P_2\}$.  
The associated word is therefore $w=
s_1s_2s_3s_2$.  
Conversely, starting with an element
$w$ of length $t$, we describe explicitly how to find the $t$ 
indecomposable $kQ$-modules which are not in the corresponding quotient closed
subcategory.  

Our method for proving this correspondence is to work with the preprojective
algebra $\Pi$ associated to $Q$.  For each element $w$ in $W_Q$, there is an
associated ideal $I_w$ in $\Pi$ (see \cite{IR,BIRS}), such that 
$\Pi/I_w$ is a finite dimensional algebra.  We associate to $I_w$ the 
subcategory $\C(I_w)= \add ((I_w)_{kQ}) \cap \mod kQ$.
This is a subcategory of $\mathcal P$, the preprojective $kQ$-modules.
We show that the additive category generated by 
$\C(I_w)$ together with the regular and preinjective $kQ$-modules is quotient closed and coincides with the subcategory corresponding
as above to $w$ in $W_Q$; we also show that any cofinite quotient 
closed subcategory of $\mod kQ$ is of this form.  In our proofs, we have to
distinguish between the Dynkin and non-Dynkin cases, with the Dynkin case 
being the more complicated one.  To get the flavour of our results, the reader
might prefer on a first reading to skip Sections~\ref{sectionfive} and \ref{sectionsix}, which deal with the Dynkin case.

For the most part, in this paper, we work over an algebraically closed ground
field.  This is necessary because of our reliance on 
\cite{BIRS}, which makes this assumption.  However, using the 
technology of \emph{Frobenius maps}, we show
that our main result extends to arbitrary finite dimensional 
hereditary algebras over finite
fields.  

Another interesting subcategory of $\mod kQ$ associated to an element 
$w$ in $W_Q$ is $\C(\Pi/I_w)$.  When $Q$ is Dynkin, we use our main
theorem to show that the map from $w$ to $\C(\Pi/I_w)$ is a bijection from
$W_Q$ to the subclosed subcategories of $\mod kQ$.  In the general
case, we conjecture that there is a correspondence between the elements of 
$W_Q$ and a certain specific subclass of the subclosed
subcategories containing finitely many indecomposables.  

The correspondence $w \to \C(\Pi/I_w)$ was already investigated in
a special case in \cite{AIRT}.  It was shown that this gives 
a bijection between a special class of elements of $W_Q$ called 
$c$-sortable \cite{Re} and torsion-free subcategories of $\mod kQ$ with a 
finite number of indecomposable objects.  (Such a bijection had previously
been constructed in \cite[Theorem 4.3]{IT}.)

By analogy, it would be interesting to describe the elements $w$ of 
$W_Q$ such that $\C(I_w)$ is a torsion class.  Also, given such an element
$w$, one might wish to determine the element of $W_Q$ corresponding to the
associated torsion-free class.  We solve these problems for finite type, 
and we state a conjecture for the general case.

Quotient closed subcategories have not been extensively studied previously,
though torsion classes are an important and well-studied special case.  The
dual concept, that of subclosed subcategories, arises in recent work of Ringel
\cite{r3} and of Krause and Prest \cite{KP}. 
%While torsion-free classes have been studied a lot, thsi is not the case for
%the more general class of subclosed subcategories, except for some recent
%interesting work by Ringel \cite{ri}.  
In particular, he has dealt 
with subclosed subcategories of infinite type, and has shown that any infinite
subclosed subcategory of finite dimensional modules over a finite 
dimensional algebra contains a minimal infinite such subcategory.  His 
work was motivated by previous work on the Gabriel-Roiter measure.  

The inclusion order on cofinite quotient-closed subcategories transfers
over to give a partial order on the Weyl group $W$.  This partial order
was first studied by Armstrong \cite{Arm}, under the name of ``sorting order''.  Part of our motivation for
this paper was to provide a representation-theoretic interpretation for
this family of partial orders.  

One of the referees raised a similar question: how could one transfer the
group structure on $W$ to the collection of cofinite quotient-closed
subcategories?  This is an interesting question, which we are not able
to answer.  

The paper is organized as follows. In Section~2 we give some essential
background material, and we state our main theorem giving a bijection
between elements of $W_Q$ and cofinite, quotient closed subcategories 
of $\mod kQ$. 
In Section~3 we establish preliminary results on subfunctors of $\Ext_{\Pi}^1$
and on interpretations of reflection functors, which are important for the 
proof of the main result.   
We show in Section~4 that $I_w$ is determined by $\C(I_w)$ in the non Dynkin case, and in Section~5 that $\Pi/I_w$ is determined by 
$\C(\Pi/I_w)$ in the Dynkin case. 
%\comment{Maybe rephrase the previous
%sentence to make it more precise/accurate.}
In Section~6 we give some more results when $Q$ is a Dynkin quiver, including the relationship between subcategories of the form $\C(\Pi/I_w)$ and those of the form $\C(I_w).$ The proof of the main 
theorem is given in Section~7. In Section~8 we extend our main theorem to give
a bijection between arbitrary quotient closed subcategories of $\mathcal P$
and a suitable class of possibly infinite words. 
In Section~9 we extend our main theorem to cover cofinite quotient closed
subcategories of representations of hereditary algebras which are 
finite dimensional over a finite field.
In Section~10 we deal with the categories $\C(\Pi/I_w),$ and show that they are exactly the subclosed subcategories in the Dynkin case. We also state a conjecture for the non-Dynkin case. 
In Section~11 we investigate when the categories $\C(I_w)$ are torsion classes, and how to describe the associated torsion-free classes in that case. 
We give a complete answer in the Dynkin case, and state a conjecture in the general case.
In Section~12 we show how our main theorem can be used to recover the
characterization by Postnikov \cite{Postnikov} in terms of $\le$-diagrams 
of the leftmost reduced subwords (equivalently, positive
distinguished subexpressions)
in the type $A$ Grassmannian
permutations. 
We discuss 
the connections to the work of Armstrong, mentioned above, in Section~13. 

\subsection*{Acknowledgements} 
SO was supported by FRINAT grants 19660 and 221893 from the Norwegian
Research Council.   
IR was supported by FRINAT grant 19660 from the Norwegian Research Council.  
HT was supported by an NSERC Discovery Grant.  
He also gratefully acknowledges 
the hospitality of the Mathematics Institute at NTNU.  
We thank the referees for their helpful comments,
and especially for asking what could be said over ground fields which are not
algebraically closed.  

\section{Statement of main results}
In this section we state our main results, and give relevant background
material and an example for illustration.  

Let $Q$ be a quiver without oriented cycles and with vertices $1,\dots,n$,
and let $k$ be an algebraically closed field.  Denote by $kQ$ the associated
path algebra.  The Weyl group $W=W_Q$ associated to $Q$ has a
distinguished set of generators $s_1,\dots,s_n$, with relations 
$s_i^2=e$ (the identity element), $s_is_j=s_js_i$ if there is no 
arrow between $i$ and $j$, and $s_is_js_i=s_js_is_j$ if there 
is exactly one arrow between $i$ and $j$.  For an element $w$ in $ W$, an
expression $\underline w=s_{i_1}\dots s_{i_t}$ (called a \emph{word}) is said
to be \emph{reduced} if $t$ is as small as possible.  In this case, $t=\ell(w)$
is the \emph{length} of $w$.  

Our main result is the following:

\begin{theorem} \label{two.one} There is a natural bijection between the elements in the
Weyl group $W_Q$ and the cofinite (additive) quotient closed subcategories
of the category $\mod kQ$ of finitely generated $kQ$-modules.  
\end{theorem}

The following observation shows that we can equally well consider the 
cofinite quotient closed subcategories of the category $\mathcal P$ of
preprojective $kQ$-modules.  

\begin{proposition} Any cofinite quotient closed subcategory of $\mod kQ$ 
contains all the non-preprojective indecomposable $kQ$-modules.  
Further, any cofinite quotient closed subcategory of $\mathcal P$ can be 
extended to a cofinite quotient closed subcategory of $\mod kQ$ by taking 
the additive subcategory generated by it 
together with all the non-preprojective indecomposable 
$kQ$-modules.
\end{proposition}

\begin{proof}
For $Q$ Dynkin, $\mathcal P=\mod kQ$, so there is nothing to prove.  
Assume that $Q$ is not Dynkin, and let ${\mathcal B}$ be an additive 
cofinite quotient closed subcategory of $\mod kQ$.  Since $\mathcal 
B$ is cofinite, $\tau^{-i} kQ$ is in $\mathcal B$ for $i$ sufficiently
large.  Since $\mathcal B$ is quotient closed and $\tau^{-1}$ preserves
epimorphisms, it follows that $\mathcal B$ contains the regular and 
preinjective indecomposables of $\mod kQ$.  This proves the first point.  

Now suppose that $\mathcal A$ is a cofinite, quotient closed subcategory
of $\mathcal P$.  Let $\overline {\mathcal A}$ be the additive subcategory
of $\mod kQ$ generated by $\mathcal A$ together with all the non-preprojective
indecomposable objects of $\mod kQ$.  Clearly, $\overline {\mathcal A}$ is
cofinite in $\mod kQ$, and it is quotient closed because there are no
non-zero 
maps from a regular or preinjective module to an object of $\mathcal P$.
\end{proof}

We introduce some more terminology in order to state the main theorem
more explicitly.  Let $\mathcal A$ be a cofinite, quotient closed
subcategory of $\mathcal P$, and let $\mathcal X$ be the finite set of 
indecomposable preprojective modules not in $\mathcal A$.  Let 
$P_1,\dots,P_n$ be an ordering of the indecomposable projective $kQ$-modules,
compatible with the orientation of $Q$, that is, such that if
$i<j$ then $\Hom(P_j,P_i)=0$.  Consider the ordering 
$P_1,\dots,P_n,\tau^{-1}P_1,\dots,\tau^{-1}P_n,\tau^{-2}P_1,\dots$ of the
indecomposable preprojective $kQ$-modules, dropping any $\tau^{-i}P_j$ which
are zero.  

From this, we get an induced ordering on $\mathcal X$.  We replace each
module in $\mathcal X$ of the form $\tau^{-i}P_j$ for some $i$, by $s_j$,
thereby obtaining a word $\underline w$ associated to the subcategory
$\mathcal A$.  

Conversely, start with an element $w\in W_Q$. Consider the infinite
word ${\underline c}^\infty=\underline c\, \underline c\, \underline c\dots$,
where $\underline c=s_1\dots s_n$ is what is called a \emph{Coxeter element}.
We match the reflections in $\underline c^\infty$ with the indecomposable
objects in $\mathcal P$, so that the first $s_i$ corresponds to $P_i$,
the second to $\tau^{-1}P_i$, and so on.  Now, among all the reduced 
expressions $s_{i_1}\dots s_{i_t}$ in $\underline c^\infty$ for $w$ , 
we choose the leftmost one, in the sense that $s_{i_1}$ is as far to the left
as possible in $\underline c^{\infty}$, and, among such expressions, $s_{i_2}$ is
as far to the left as possible (but to the right of $s_{i_1}$), and so on
for each $s_{i_j}$.  In this way, we determine a unique subword
$\underline w$ of $\underline c^{\infty}$.  Consider the associated set 
$\mathcal X$ of indecomposable preprojective modules 
corresponding to this subword, as discussed above.
Then we associate to $w$ the additive subcategory $\mathcal A$ of $\mathcal P$, whose
indecomposable objects are the indecomposable objects of $\mathcal P$ 
which do not lie in $\mathcal X$.  We can now state the following more 
explicit version of our main theorem:

\begin{theorem} \label{two.two}
There is a bijective correspondence between elements $w\in W_Q$ and
cofinite quotient closed subcategories of $\mathcal P$, which
can be described as follows:
\begin{itemize}
\item[(a)] The correspondence $w \to \mathcal A$ is given by removing
from $\mathcal P$ the indecomposable modules corresponding to the 
leftmost word $\underline w$ for $w$ in $\underline c^\infty$.  
\item[(b)] The correspondence $\mathcal A \to w$ is given by taking 
the finite set $\mathcal X$ of indecomposable preprojective modules
not in $\mathcal A$, and associating to it a word as described above.
\end{itemize}
\end{theorem}

In order to prove these results, we work with the preprojective algebra
$\Pi=\Pi_Q$ associated to $kQ$.  For each arrow $a$ in $kQ$, add an arrow
$a^*$ in the opposite direction to get a new quiver $\overline Q$.  
Then, by definition, 
$$\Pi = k\overline Q/ \sum_{a} (aa^*-a^*a).$$
We write $\mod \Pi$ for the category of finitely generated right $\Pi$-modules,
and $\Mod \Pi$ for the category of all right $\Pi$-modules.  

Let $e_i$ be the idempotent corresponding to the vertex $i$.  Then
consider the ideal $I_i= \Pi (1-e_i) \Pi$ in $\Pi$.  When
$\underline w=s_{i_1}\dots s_{i_t}$ is a reduced expression for
$w \in W_Q$, then $I_w$ is (well-)defined by $I_w=I_{i_t}\dots I_{i_1}$
\cite{BIRS}.  (Note that the product of ideals is taken in the opposite
order to the product of reflections in $\underline w$.  This follows
the convention of \cite{AIRT}.)  Any $\Pi$-module, like $I_w$, is a $kQ$-module
by restriction.  

Consider the subcategory $\C(I_w)$ of $\mod \Pi$, whose indecomposable modules
are those which appear as indecomposable summands of $I_w$ as a $kQ$-module.
We then have the following:

\begin{theorem} \label{two.three}
The cofinite quotient closed subcategory of $\mathcal P$ which corresponds
to $w \in W_Q$ under the bijection of Theorem~\ref{two.two} can also
be described as $\C(I_w)$.  
\end{theorem}

We now illustrate the
above results.

\begin{example}
Let $Q$ be the following quiver:
\[ \begin{tikzpicture}[xscale=2]
\node (A) at (0,0) {$1$};
\node (B) at (1,0) {$2$};
\node (C) at (2,0) {$3$}; 
\draw [->] (A) ..  controls (1,-.5).. (C);
\draw [->] (B) -- (C);
\draw [->] (A) -- (B);
\end{tikzpicture} \]

Then the indecomposable projective $kQ$-modules are 
$$P_1=1, \qquad P_2=\begin{array}[t]{c}2\\1\end{array}, \qquad 
P_3=\begin{array}[t]{cccc}
& & 3.\!\\&2& & 1\\1\end{array}$$

Let $w=s_1s_2s_3s_2s_1$.  
We have $I_w= I_1I_2I_3I_2I_1$ and 
$I_w= \widetilde P_1I_w 
\oplus \widetilde P_2I_w
\oplus \widetilde P_3I_w$,
where 
$\widetilde P_1,
\widetilde P_2, \widetilde P_3$ 
are the indecomposable 
projective $\Pi$-modules.
The modules $\widetilde P_i$, together with their submodules 
$\widetilde P_iI_w$, are illustrated below.  The regions in grey indicate
the parts that do not appear in $I_w$.  Solid lines indicate what remains
connected upon restriction to $kQ$.  
 
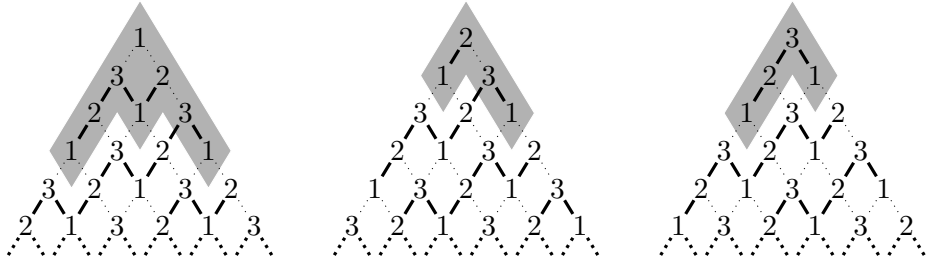
\[ \begin{tikzpicture}[xscale=.3,yscale=-.5]
 \draw [fill=black!30,draw=white] (0,-1) -- (4,3) -- (3,4) -- (1,2) -- (0,3) -- (-1,2) -- (-3,4) -- (-4,3) -- cycle;
 \node (A) at (0,0) [inner sep=1pt,outer sep=0pt] {$1$};
 \node (B) at (-1,1) [inner sep=1pt,outer sep=0pt] {$3$};
 \node (C) at (1,1) [inner sep=1pt,outer sep=0pt] {$2$};
 \node (D) at (-2,2) [inner sep=1pt,outer sep=0pt] {$2$};
 \node (E) at (0,2) [inner sep=1pt,outer sep=0pt] {$1$};
 \node (F) at (2,2) [inner sep=1pt,outer sep=0pt] {$3$};
 \node (G) at (-3,3) [inner sep=1pt,outer sep=0pt] {$1$};
 \node (H) at (-1,3) [inner sep=1pt,outer sep=0pt] {$3$};
 \node (I) at (1,3) [inner sep=1pt,outer sep=0pt] {$2$};
 \node (J) at (3,3) [inner sep=1pt,outer sep=0pt] {$1$};
 \node (K) at (-4,4) [inner sep=1pt,outer sep=0pt] {$3$};
 \node (L) at (-2,4) [inner sep=1pt,outer sep=0pt] {$2$};
 \node (M) at (0,4) [inner sep=1pt,outer sep=0pt] {$1$};
 \node (N) at (2,4) [inner sep=1pt,outer sep=0pt] {$3$};
 \node (O) at (4,4) [inner sep=1pt,outer sep=0pt] {$2$};
 \node (P) at (-5,5) [inner sep=1pt,outer sep=0pt] {$2$};
 \node (Q) at (-3,5) [inner sep=1pt,outer sep=0pt] {$1$};
 \node (R) at (-1,5) [inner sep=1pt,outer sep=0pt] {$3$};
 \node (S) at (1,5) [inner sep=1pt,outer sep=0pt] {$2$};
 \node (T) at (3,5) [inner sep=1pt,outer sep=0pt] {$1$};
 \node (U) at (5,5) [inner sep=1pt,outer sep=0pt] {$3$};
 \draw [dotted] (A) -- (B);
 \draw [dotted] (A) -- (C);
 \draw [very thick] (B) -- (D);
 \draw [very thick] (B) -- (E);
 \draw [very thick] (C) -- (E);
 \draw [dotted] (C) -- (F);
 \draw [very thick] (D) -- (G);
 \draw [dotted] (D) -- (H);
 \draw [dotted] (E) -- (H);
 \draw [dotted] (E) -- (I);
 \draw [very thick] (F) -- (I);
 \draw [very thick] (F) -- (J);
 \draw [dotted] (G) -- (K);
 \draw [dotted] (G) -- (L);
 \draw [very thick] (H) -- (L);
 \draw [very thick] (H) -- (M);
 \draw [very thick] (I) -- (M);
 \draw [dotted] (I) -- (N);
 \draw [dotted] (J) -- (N);
 \draw [dotted] (J) -- (O);
 \draw [very thick] (K) -- (P);
 \draw [very thick] (K) -- (Q);
 \draw [very thick] (L) -- (Q);
 \draw [dotted] (L) -- (R);
 \draw [dotted] (M) -- (R);
 \draw [dotted] (M) -- (S);
 \draw [very thick] (N) -- (S);
 \draw [very thick] (N) -- (T);
 \draw [very thick] (O) -- (T);
 \draw [dotted] (O) -- (U);
 \draw [dotted,very thick] (P) -- (-5.8,5.8);
 \draw [dotted,very thick] (P) -- (-4.2,5.8);
 \draw [dotted,very thick] (Q) -- (-3.8,5.8);
 \draw [dotted,very thick] (Q) -- (-2.2,5.8);
 \draw [dotted,very thick] (R) -- (-1.8,5.8);
 \draw [dotted,very thick] (R) -- (-0.2,5.8);
 \draw [dotted,very thick] (S) -- (0.2,5.8);
 \draw [dotted,very thick] (S) -- (1.8,5.8);
 \draw [dotted,very thick] (T) -- (2.2,5.8);
 \draw [dotted,very thick] (T) -- (3.8,5.8);
 \draw [dotted,very thick] (U) -- (4.2,5.8);
 \draw [dotted,very thick] (U) -- (5.8,5.8);
\end{tikzpicture}
\qquad
\begin{tikzpicture}[xscale=.3,yscale=-.5]
 \draw [fill=black!30,draw=white] (0,-1) -- (3,2) -- (2,3) -- (0,1) -- (-1,2) -- (-2,1) -- cycle;
 \node (A) at (0,0) [inner sep=1pt,outer sep=0pt] {$2$};
 \node (B) at (-1,1) [inner sep=1pt,outer sep=0pt] {$1$};
 \node (C) at (1,1) [inner sep=1pt,outer sep=0pt] {$3$};
 \node (D) at (-2,2) [inner sep=1pt,outer sep=0pt] {$3$};
 \node (E) at (0,2) [inner sep=1pt,outer sep=0pt] {$2$};
 \node (F) at (2,2) [inner sep=1pt,outer sep=0pt] {$1$};
 \node (G) at (-3,3) [inner sep=1pt,outer sep=0pt] {$2$};
 \node (H) at (-1,3) [inner sep=1pt,outer sep=0pt] {$1$};
 \node (I) at (1,3) [inner sep=1pt,outer sep=0pt] {$3$};
 \node (J) at (3,3) [inner sep=1pt,outer sep=0pt] {$2$};
 \node (K) at (-4,4) [inner sep=1pt,outer sep=0pt] {$1$};
 \node (L) at (-2,4) [inner sep=1pt,outer sep=0pt] {$3$};
 \node (M) at (0,4) [inner sep=1pt,outer sep=0pt] {$2$};
 \node (N) at (2,4) [inner sep=1pt,outer sep=0pt] {$1$};
 \node (O) at (4,4) [inner sep=1pt,outer sep=0pt] {$3$};
 \node (P) at (-5,5) [inner sep=1pt,outer sep=0pt] {$3$};
 \node (Q) at (-3,5) [inner sep=1pt,outer sep=0pt] {$2$};
 \node (R) at (-1,5) [inner sep=1pt,outer sep=0pt] {$1$};
 \node (S) at (1,5) [inner sep=1pt,outer sep=0pt] {$3$};
 \node (T) at (3,5) [inner sep=1pt,outer sep=0pt] {$2$};
 \node (U) at (5,5) [inner sep=1pt,outer sep=0pt] {$1$};
 \draw [very thick] (A) -- (B);
 \draw [dotted] (A) -- (C);
 \draw [dotted] (B) -- (D);
 \draw [dotted] (B) -- (E);
 \draw [very thick] (C) -- (E);
 \draw [very thick] (C) -- (F);
 \draw [very thick] (D) -- (G);
 \draw [very thick] (D) -- (H);
 \draw [very thick] (E) -- (H);
 \draw [dotted] (E) -- (I);
 \draw [dotted] (F) -- (I);
 \draw [dotted] (F) -- (J);
 \draw [very thick] (G) -- (K);
 \draw [dotted] (G) -- (L);
 \draw [dotted] (H) -- (L);
 \draw [dotted] (H) -- (M);
 \draw [very thick] (I) -- (M);
 \draw [very thick] (I) -- (N);
 \draw [very thick] (J) -- (N);
 \draw [dotted] (J) -- (O);
 \draw [dotted] (K) -- (P);
 \draw [dotted] (K) -- (Q);
 \draw [very thick] (L) -- (Q);
 \draw [very thick] (L) -- (R);
 \draw [very thick] (M) -- (R);
 \draw [dotted] (M) -- (S);
 \draw [dotted] (N) -- (S);
 \draw [dotted] (N) -- (T);
 \draw [very thick] (O) -- (T);
 \draw [very thick] (O) -- (U);
 \draw [dotted,very thick] (P) -- (-5.8,5.8);
 \draw [dotted,very thick] (P) -- (-4.2,5.8);
 \draw [dotted,very thick] (Q) -- (-3.8,5.8);
 \draw [dotted,very thick] (Q) -- (-2.2,5.8);
 \draw [dotted,very thick] (R) -- (-1.8,5.8);
 \draw [dotted,very thick] (R) -- (-0.2,5.8);
 \draw [dotted,very thick] (S) -- (0.2,5.8);
 \draw [dotted,very thick] (S) -- (1.8,5.8);
 \draw [dotted,very thick] (T) -- (2.2,5.8);
 \draw [dotted,very thick] (T) -- (3.8,5.8);
 \draw [dotted,very thick] (U) -- (4.2,5.8);
 \draw [dotted,very thick] (U) -- (5.8,5.8);
\end{tikzpicture}
\qquad
\begin{tikzpicture}[xscale=.3,yscale=-.5]
 \draw [fill=black!30,draw=white] (0,-1) -- (2,1) -- (1,2) -- (0,1) -- (-2,3) -- (-3,2) -- cycle;
 \node (A) at (0,0) [inner sep=1pt,outer sep=0pt] {$3$};
 \node (B) at (-1,1) [inner sep=1pt,outer sep=0pt] {$2$};
 \node (C) at (1,1) [inner sep=1pt,outer sep=0pt] {$1$};
 \node (D) at (-2,2) [inner sep=1pt,outer sep=0pt] {$1$};
 \node (E) at (0,2) [inner sep=1pt,outer sep=0pt] {$3$};
 \node (F) at (2,2) [inner sep=1pt,outer sep=0pt] {$2$};
 \node (G) at (-3,3) [inner sep=1pt,outer sep=0pt] {$3$};
 \node (H) at (-1,3) [inner sep=1pt,outer sep=0pt] {$2$};
 \node (I) at (1,3) [inner sep=1pt,outer sep=0pt] {$1$};
 \node (J) at (3,3) [inner sep=1pt,outer sep=0pt] {$3$};
 \node (K) at (-4,4) [inner sep=1pt,outer sep=0pt] {$2$};
 \node (L) at (-2,4) [inner sep=1pt,outer sep=0pt] {$1$};
 \node (M) at (0,4) [inner sep=1pt,outer sep=0pt] {$3$};
 \node (N) at (2,4) [inner sep=1pt,outer sep=0pt] {$2$};
 \node (O) at (4,4) [inner sep=1pt,outer sep=0pt] {$1$};
 \node (P) at (-5,5) [inner sep=1pt,outer sep=0pt] {$1$};
 \node (Q) at (-3,5) [inner sep=1pt,outer sep=0pt] {$3$};
 \node (R) at (-1,5) [inner sep=1pt,outer sep=0pt] {$2$};
 \node (S) at (1,5) [inner sep=1pt,outer sep=0pt] {$1$};
 \node (T) at (3,5) [inner sep=1pt,outer sep=0pt] {$3$};
 \node (U) at (5,5) [inner sep=1pt,outer sep=0pt] {$2$};
 \draw [very thick] (A) -- (B);
 \draw [very thick] (A) -- (C);
 \draw [very thick] (B) -- (D);
 \draw [dotted] (B) -- (E);
 \draw [dotted] (C) -- (E);
 \draw [dotted] (C) -- (F);
 \draw [dotted] (D) -- (G);
 \draw [dotted] (D) -- (H);
 \draw [very thick] (E) -- (H);
 \draw [very thick] (E) -- (I);
 \draw [very thick] (F) -- (I);
 \draw [dotted] (F) -- (J);
 \draw [very thick] (G) -- (K);
 \draw [very thick] (G) -- (L);
 \draw [very thick] (H) -- (L);
 \draw [dotted] (H) -- (M);
 \draw [dotted] (I) -- (M);
 \draw [dotted] (I) -- (N);
 \draw [very thick] (J) -- (N);
 \draw [very thick] (J) -- (O);
 \draw [very thick] (K) -- (P);
 \draw [dotted] (K) -- (Q);
 \draw [dotted] (L) -- (Q);
 \draw [dotted] (L) -- (R);
 \draw [very thick] (M) -- (R);
 \draw [very thick] (M) -- (S);
 \draw [very thick] (N) -- (S);
 \draw [dotted] (N) -- (T);
 \draw [dotted] (O) -- (T);
 \draw [dotted] (O) -- (U);
 \draw [dotted,very thick] (P) -- (-5.8,5.8);
 \draw [dotted,very thick] (P) -- (-4.2,5.8);
 \draw [dotted,very thick] (Q) -- (-3.8,5.8);
 \draw [dotted,very thick] (Q) -- (-2.2,5.8);
 \draw [dotted,very thick] (R) -- (-1.8,5.8);
 \draw [dotted,very thick] (R) -- (-0.2,5.8);
 \draw [dotted,very thick] (S) -- (0.2,5.8);
 \draw [dotted,very thick] (S) -- (1.8,5.8);
 \draw [dotted,very thick] (T) -- (2.2,5.8);
 \draw [dotted,very thick] (T) -- (3.8,5.8);
 \draw [dotted,very thick] (U) -- (4.2,5.8);
 \draw [dotted,very thick] (U) -- (5.8,5.8);
\end{tikzpicture} \]

Then we compute that 
$(\widetilde P_1I_w)_{kQ}
= \tau^{-1} P_3 
\oplus \left(\bigoplus_{i=3}^\infty \tau^{-i} P_1\right)$,
$(\widetilde P_2I_w)_{kQ}
=\tau^{-1}P_1 \oplus 
\left(\bigoplus_{i=2}^\infty
\tau^{-i}P_2\right)$,
and 
$(\widetilde P_3I_w)_{kQ}
= \bigoplus_{i=1}^\infty 
\tau^{-i} P_3$.  We see that the indecomposable $kQ$-modules not 
in $\mathscr C(I_w)$ are 
$P_1,P_2,P_3,\tau^{-1}P_2,\tau^{-2}P_1$.  

We also illustrate how to see this by
using our direct description of 
the missing set of $\ell(w)=5$ indecomposable
$kQ$-modules.  Consider the infinite word
${\underline c}^\infty=
s_1s_2s_3s_1s_2s_3s_1s_2s_3\dots$.
We indicate the leftmost subword for the element $w$ by 
underlining the corresponding $s_i$:
$\underline s_1 \underline s_2 \underline s_3 s_1 
\underline s_2 s_3 \underline s_1 \dots$.
Hence we obtain the associated set of indecomposable modules
$P_1,P_2,P_3,\tau^{-1} P_2, \tau^{-2} P_1$.\end{example}

\section{Results on preprojective algebras}
In this section we give two results, which will be useful later, on the 
relationship between the path algebra $kQ$ and the associated preprojective
algebra $\Pi$.  The first gives a long exact sequence involving the subfunctor
of the ordinary $\Ext_\Pi^1$ functor, given by short exact sequences of $\Pi$-modules which split as $kQ$-modules.  The second one gives a comparison between the
functor $-\otimes_\Pi I_i$ and APR-tilting for $kQ$ at the vertex $i$.  
Similar statements appear as \cite[Corollary 2.11]{AIRT}, \cite[Proposition 22]
{BK}.

Relative homological algebra was investigated by Auslander and Solberg in \cite{AS}. They consider certain subfunctors of $\Ext_\Pi^1$ given by a choice of short exact sequences. In our context, we will be interested in those short exact sequences of $\Pi$-modules which split upon restriction to $kQ$. We write $\underline{\Ext}^1_{\Pi}$ for the subfunctor of $\Ext^1_{\Pi}$ given by these short exact sequences.

In the following lemma, we use a description of preprojective algebras which first appeared in \cite[Proposition 3.1]{BGL}:
\[ \Pi = {\rm T}_{kQ} \Omega \quad \text{with} \quad \Omega = \Ext_{kQ}^1(D(kQ), kQ). \]
Here ${\rm T}_{kQ}$ denotes the tensor algebra over $kQ$, that is
\[ {\rm T}_{kQ} \Omega = kQ \oplus \Omega \oplus (\Omega \otimes_{kQ} \Omega) \oplus \Omega^{\otimes 3} \oplus \cdots . \]
In this description, a $\Pi$-module is given by a $kQ$-module $M$ and a multiplication rule $\varphi_M \colon M \otimes_{kQ} \Omega \to M$.

One may note that for finite dimensional $M$ we have $M \otimes_{kQ} \Omega = \tau^{-1} M$, so in this case the above description coincides with Ringel's \cite{Ringel}.

\begin{lemma} \label{lemma.exact_1}
For two $\Pi$-modules $(A, \varphi_A)$ and $(B, \varphi_B)$ we have an exact sequence
\[ 0 \to \Hom_{\Pi}(A, B) \to^f \Hom_{kQ}(A, B) \to^g \Hom_{kQ}(A \otimes_{kQ} \Omega, B) \to^h \underline{\Ext}^1_{\Pi}(A, B) \to 0, \]
with maps given by
\begin{align*}
f \colon & \,\,\text{restriction functor} \\
g \colon & \,\,g(\alpha) = \varphi_B \circ (\alpha \otimes 1_\Omega) - \alpha \circ \varphi_A \\
h \colon & \,\,h(\beta) \text{ is given by the $kQ$-split short exact sequence}\\
& 0 \to (B, \varphi_B) \to (A \oplus B, \begin{pmatrix} \varphi_A & \\ \beta & \varphi_B \end{pmatrix} ) \to (A, \varphi_A) \to 0
\end{align*}
\end{lemma}

\begin{proof}
Injectivity of $f$ and surjectivity of $h$ are clear.

It follows from the definition of $g$ that $\Ker g = \Im f$.

We determine $\Ker h$: $\beta \in \Ker h$ if and only if the following diagram can be completed commutatively:
\[ \begin{tikzpicture}[xscale=2,yscale=-1.8]
 \node (X) at (0,0) {$0$};
 \node (Y) at (0,1) {$0$};
 \node (A) at (1,0) {$(B, \varphi_B)$};
 \node (B) at (1,1) {$(B, \varphi_B)$};
 \node (C) at (3,0) {$(A \oplus B, \begin{pmatrix} \varphi_A &  \\ & \varphi_B \end{pmatrix} )$};
 \node (D) at (3,1) {$(A \oplus B, \begin{pmatrix} \varphi_A & \\ \beta & \varphi_B \end{pmatrix})$};
 \node (E) at (5,0) {$(A, \varphi_A)$};
 \node (F) at (5,1) {$(A, \varphi_A)$};
 \node (Z) at (6,0) {$0$};
 \node (W) at (6,1) {$0$};
 \draw [double distance=1.5pt] (A) -- (B);
 \draw [->, dashed] (C) -- node [right] {$\Psi$} (D);
 \draw [double distance=1.5pt] (E) -- (F);
 \draw [->] (A) -- (C);
 \draw [->] (C) -- (E);
 \draw [->] (B) -- (D);
 \draw [->] (D) -- (F);
 \draw [->] (X) -- (A);
 \draw [->] (Y) -- (B);
 \draw [->] (E) -- (Z);
 \draw [->] (F) -- (W);
\end{tikzpicture} \]
i.e.\ there is some $\Psi \in \End_{kQ}(A \oplus B)$ such that
\begin{itemize}
\item $\Psi \circ \begin{pmatrix} \varphi_A & \\ & \varphi_B \end{pmatrix} = \begin{pmatrix} \varphi_A &  \\ \beta & \varphi_B \end{pmatrix} \circ (\Psi \otimes 1_\Omega)$ ($\Psi$ is a morphism of $\Pi$-modules),
\item $\Psi \circ \begin{pmatrix} 0 \\ 1_B \end{pmatrix} = \begin{pmatrix} 0 \\ 1_B \end{pmatrix}$ (the left square commutes), and
\item $(1_A \; 0) \circ \Psi = (1_A \; 0)$ (the right square commutes).
\end{itemize}
Writing $\Psi = \begin{pmatrix} \Psi_{AA} & \Psi_{BA} \\ \Psi_{AB} & \Psi_{BB} \end{pmatrix}$ the latter two points amount to $\Psi_{AA} = 1_A$, $\Psi_{BB} = 1_B$, and $\Psi_{BA} = 0$. Thus the first one becomes
\[  \begin{pmatrix} 1_A & \\ \Psi_{AB} & \ 1_B \end{pmatrix} \begin{pmatrix} \varphi_A & \\ & \varphi_B \end{pmatrix} = \begin{pmatrix} \varphi_A & \\ \beta & \varphi_B \end{pmatrix} \begin{pmatrix} 1_A \otimes 1_\Omega & \\ \Psi_{AB} \otimes 1_\Omega & \ 1_B \otimes 1_\Omega \end{pmatrix}, \]
that is
\[ \Psi_{AB} \circ \varphi_A = \beta + \varphi_B \circ (\Psi_{AB} \otimes 1). \]
Hence we have $\beta = g(- \Psi_{AB}) \in \Im g$.

The same calculation read backwards shows that $h \circ g = 0$.
\end{proof}

\begin{proposition}\label{three.two}
Let $0\to A \to B \to C\to 0$ be a short exact sequence of $\Pi$-modules which splits upon restriction to ${kQ}$. Then for any $X \in \mod \Pi$ there are induced exact sequences
\begin {eqnarray*}
& 0 \to \Hom_{\Pi}(X, A) \to \Hom_{\Pi}(X, B) \to \Hom_{\Pi}(X, C) \\
&\qquad\qquad \to \underline{\Ext}^1_{\Pi}(X, A) \to \underline{\Ext}^1_{\Pi}(X, B) \to \underline{\Ext}^1_{\Pi}(X, C)  \to 0
\end{eqnarray*}
and
\begin{eqnarray*}
&0 \to \Hom_{\Pi}(C, X) \to \Hom_{\Pi}(B, X) \to \Hom_{\Pi}(A, X) \\
&\qquad \qquad \to \underline{\Ext}^1_{\Pi}(C, X) \to \underline{\Ext}^1_{\Pi}(B, X) \to \underline{\Ext}^1_{\Pi}(A, X) \to 0.
\end{eqnarray*}
\end{proposition}

\begin{proof}
Note that since $0\to A \to B \to C \to 0$ is split exact over ${kQ}$ the sequences
\begin{align*}
& 0\to \Hom_{kQ}(C, X) \to \Hom_{kQ}(B, X) \to \Hom_{kQ}(A, X) \to 0\text{ and} \\
& 0\to \Hom_{kQ}(C \otimes_{kQ} \Omega, X) \to \Hom_{kQ}(B \otimes_{kQ} \Omega, X) \to \Hom_{kQ}(A \otimes_{kQ} \Omega, X)
\to 0
\end{align*}
are exact. Hence the proposition follows from Lemma~\ref{lemma.exact_1} and the snake lemma.
\end{proof}

Now we investigate the interaction of tensoring with an ideal $I_i$ and restricting to $kQ$. It turns out that on the level of $kQ$-modules, tensoring with $I_i$ corresponds to applying an APR-tilt. A similar observation had already
been made in \cite{AIRT}. We start by recalling the notion of APR-tilting
\cite{APR}, which is a module-theoretic interpretation of the reflections
of \cite{BGP}.

Let $Q$ be a (connected, acyclic, finite) quiver, and $i$ be a source of $Q$, so the corresponding indecomposable projective $kQ$-module $P_i$ is simple.

The $kQ$-module
\[ T = \tau^{-1} P_i \oplus kQ / P_i \]
is an APR-tilting module. We set ${Q'}$ to be the Gabriel quiver of 
$\End_{kQ}(T)$, so that 
$k{Q'} = \End_{kQ}(T)$. Then we have the mutually inverse equivalences
\begin{equation} \label{eq.tilting_equiv}
\R\Hom_{{kQ}}(T, -) \colon \Db(\mod {kQ}) \arrow{<->} \Db(\mod k\widet{Q}) : - \otimes_{k{\widet{Q}}}^L T.
\end{equation}

Recall from \cite{BGL,Ringel} that we have
\begin{equation} \label{eq.Pi_kQ}
\Pi = \bigoplus_{n \geq 0} \Hom_{\Db(\mod kQ)}(kQ, \tau^{-n} kQ).
\end{equation}
Since $T$ is obtained from $kQ$ by replacing one summand by its (inverse) AR-translation we also have
\begin{equation} \label{eq.Pi_T}
\Pi = \bigoplus_{n \geq 0} \Hom_{\Db(\mod kQ)}(T, \tau^{-n} T) = \bigoplus_{n \geq 0} \Hom_{\Db(\mod k\widet{Q})}(k\widet{Q}, \tau^{-n} k\widet{Q}).
\end{equation}

\begin{lemma} \label{lemma.identify_ideal}
Via the identifications above we have isomorphisms of $\Pi$-$\Pi$-bimodules
\begin{align*}
\Pi & \cong \bigoplus_{n \in \mathbb{Z}} \Hom_{\Db(\mod kQ)}(kQ, \tau^{-n} T) \\
S_i & \cong \bigoplus_{n < 0} \Hom_{\Db(\mod kQ)}(kQ, \tau^{-n} T) \\
I_i & \cong \bigoplus_{n \geq 0} \Hom_{\Db(\mod kQ)}(kQ, \tau^{-n} T)
\end{align*}
where in all cases the term on the right gets its right $\Pi$-module structure via \eqref{eq.Pi_kQ} and its left $\Pi$-module structure via \eqref{eq.Pi_T}.
\end{lemma}

\begin{proof}
The first claim is seen similarly to the identification in \eqref{eq.Pi_T}.

For the second claim note that
\[ \bigoplus_{n < 0} \Hom_{\Db(\mod kQ)}({kQ}, \tau^{-n} T) =
\Hom_{\Db(\mod kQ)}(P_i, \tau (\tau^{-1} P_i)), \]
so this module is isomorphic to $S_i$ on both sides.

The final claim follows by looking at the short exact sequence
\begin{align*} & 0 \to \bigoplus_{n \geq 0} \Hom_{\Db(\mod kQ)}({kQ}, \tau^{-n}T) \to
\bigoplus_{n \in \mathbb{Z}} \Hom_{\Db(\mod kQ)}({kQ}, \tau^{-n}T) \\
& \qquad \to \bigoplus_{n < 0} \Hom_{\Db(\mod kQ)}({kQ}, \tau^{-n}T) \to 0
\end{align*}
of $\Pi$-$\Pi$-bimodules.
\end{proof}

\begin{theorem} \label{thm.tensoringcommutes}
Let $Q$ be a quiver with a source $i$, and let $T$ be the associated APR-tilting module as above. Then the following diagram commutes.
\[ \begin{tikzpicture}[xscale=5,yscale=-1.5]
 \node (A) at (0,0) {$\Mod \Pi$};
 \node (B) at (1,0) {$\Mod \Pi$};
 \node (C) at (0,1) {$\Mod k\widet{Q}$};
 \node (D) at (1,1) {$\Mod kQ$};
 \draw [->] (A) -- node [above] {$- \otimes_{\Pi} I_i$} (B);
 \draw [->] (A) -- node [left] {res} (C);
 \draw [->] (B) -- node [right] {res} (D);
 \draw [->] (C) -- node [above] {$- \otimes_{k\widet{Q}} T$} (D);
\end{tikzpicture} \]
\end{theorem}

\begin{proof}
By Lemma~\ref{lemma.identify_ideal} the commutativity of the diagram in the theorem is equivalent to commutativity of the following diagram.
\[ \begin{tikzpicture}[xscale=9.7,yscale=-2]
 \node (A) at (0,0) {$\displaystyle \Mod \left[ \bigoplus_{n \geq 0} \Hom_{\Db(\mod kQ)}(T, \tau^{-n} T) \right]$};
 \node (B) at (1,0) {$\displaystyle \Mod \left[ \bigoplus_{n \geq 0} \Hom_{\Db(\mod kQ)}(kQ, \tau^{-n} kQ) \right]$};
 \node (C) at (0,1) {$\Mod \End_{kQ}(T)$};
 \node (D) at (1,1) {$\Mod kQ$};
 \draw [->] (A) -- node [above=8pt] {\scalebox{.7}{$\displaystyle -
   \bigotimes \left[ \bigoplus_{n \geq 0} \Hom_{\mathcal{P}}({kQ}, \tau^{-n} T) \right]$}} (B);
 \draw [->] (A) -- node [left] {res} (C);
 \draw [->] (B) -- node [right] {res} (D);
 \draw [->] (C) -- node [above] {$- \otimes T$} (D);
\end{tikzpicture} \]
Here the restriction functors are given by restriction along the natural inclusions
\begin{align*}
& \End_{kQ}(T) \embed \bigoplus_{n \geq 0}
\Hom_{\Db(\mod kQ)}(T, \tau^{-n}T)
\text{ and} \\
kQ = & \End_{kQ}(kQ) \embed \bigoplus_{n \geq 0} \Hom_{\Db(\mod kQ)}(kQ, \tau^{-n} kQ)
\text{, respectively.} 
\end{align*}
Equivalently we may regard the restriction functors as being given by (derived) tensoring with $\bigoplus_{n \geq 0}
\Hom_{\Db(\mod kQ)}(T, \tau^{-n}T)$ and $\bigoplus_{n \geq 0}
\Hom_{\Db(\mod kQ)}({kQ}, \tau^{-n}{kQ})$, respectively.

Thus the commutativity of the diagram is equivalent to the $\Pi$-$kQ$-bimodules
\[ \bigoplus_{n \geq 0} \Hom_{\Db(\mod kQ)} (T, \tau^{-n}T)
\otimes_{k\widet{{Q}}} T \quad \text{and} \quad \bigoplus_{n \geq 0} \Hom_{\Db(\mod kQ)}({kQ}, \tau^{-n}T) \]
being isomorphic.

Clearly we have a morphism from the first to the second bimodule, which is given by evaluating. To see that this morphism is bijective it suffices to check that evaluation
\[ \Hom_{\Db(\mod kQ)}(T, X) \otimes_{k\widet{Q}} T \to \Hom_{\Db(\mod kQ)}(kQ, X) \]
is bijective for any indecomposable $X \in \add \bigoplus_{n \geq 0} \tau^{-n} T$. If $X$ is concentrated in degree $0$ then this follows from the mutually inverse equivalences in \eqref{eq.tilting_equiv}. If $X$ is concentrated in degree $-1$ or lower the right hand side vanishes, so we need to show that the term on the left also vanishes. We have $\Hom_{\Db(\mod kQ)}(T, X) = 0$ unless $X = S_i[1]$, so this is the only case to consider. But
\begin{align*}
& \Hom_{\Db(\mod kQ)}(T, S_i[1]) = S_i \text{, and}
& S_i \otimes_{k\widet{Q}} T = {\rm H}^0(S_i \otimes_{k\widet{Q}}^{\L} T) = {\rm H}^0(S_i[1]) = 0,
\end{align*}
again by the mutually inverse equivalences of \eqref{eq.tilting_equiv}.
\end{proof}

\section{Describing $I_w$ for $Q$ non-Dynkin}\label{sectionfour}

Note that $kQ$ is a subalgebra of $\Pi$ in a natural way.  
To a $\Pi$-module $X$, we associate the subcategory 
$$\C(X)=\add X_{kQ} \cap \mod kQ.$$
For $Q$ a non-Dynkin quiver, we show that for
any element $w$ in the Weyl group $W$, the $\Pi$-module $I_w$ is
uniquely determined by the category $\C(I_w)$, and even by 
$\Fac \C(I_w)$.  We use this to show, in the non-Dynkin case, that 
if $\C(I_v) \subseteq \C(I_w)$, then $I_v\subseteq I_w$.  
This is
crucial for the proof of the main theorem.  The Dynkin case of this 
result will be treated in Section~\ref{sectionsix}.

For 
$\mathcal A$ any (additive) subcategory of $\mod kQ$, we write $\Fac \mathcal A$
for the category consisting of the quotients of objects in $\mathcal A$.  
We write $\filt(\mathcal A)$ for the minimal 
subcategory of finite length $\Pi$-modules
containing the objects from $\mathcal A$ and closed under $kQ$-split
extensions.  
We make the straightforward observation that $\filt (\C(I_w))\subseteq
\filt (\C(I_v))$ iff $\C(I_w)\subseteq \C(I_v)$, since the 
$kQ$-modules in $\filt(\mathcal A)$ are exactly the objects of $\mathcal A$.  

%We set
%\begin{align*}
%\filt(\mathcal{C}_w) & = \{ M \in \mod \Pi \mid \C(M) \in \mathcal{C}_w \} \text{, and} \\
%\filt(\mathcal{C}^{\Fac}_w) & = \{ M \in \mod \Pi \mid \C(M) \in \mathcal{C}^{\Fac}_w \}.
%\end{align*}
%Note that in particular all objects of $\filt(\mathcal{C}_w)$ and 
%$\filt(\mathcal{C}^{\Fac}_w)$ are finite dimensional modules.
We set
\[ \uperp I_w = \{M \in \mod \Pi \mid \underline{\Ext}_{\Pi}^1(M, I_w) = 0 \}. \]

Throughout this section, let $Q$ be a connected quiver which is not Dynkin.  
Then we know that the ideals $I_w$ are tilting $\Pi$-modules \cite{BIRS},
and that the bounded derived category of finite length $\Pi$-modules is
2-Calabi-Yau \cite{GLS}.

\begin{proposition}
For any $w \in W$ we have $\mathscr{C}(I_w) \subseteq \uperp I_w$.
\end{proposition}

\begin{proof}
We consider $F \in \Fac I_w$ finite dimensional. Then we have an epimorphism $I_w^n \epi F$ for some $n$. Applying $\Hom_{\Pi}(I_w, -)$ and using that the projective dimension of $I_w$ is at most $1$, we obtain an epimorphism $\Ext_{\Pi}^1(I_w, I_w^n) \epi \Ext_{\Pi}^1(I_w, F)$. The first space is zero since $I_w$ is tilting, so $\Ext_{\Pi}^1(I_w, F) = 0$ as well.

Next we look at $\Ext_{\Pi}^1(F, I_w)$. The short exact sequence $0 \to I_w \to \Pi \to \Pi / I_w\to 0$, together with the fact that $\Ext_{\Pi}^i(F, \Pi) = 0$ for $i = 0,1$, gives $\Ext_{\Pi}^1(F, I_w) \cong \Hom_{\Pi}(F, \Pi / I_w)$. Thus we have
\begin{align*}
\Ext_{\Pi}^1(F, I_w) & \cong \Hom_{\Pi}(F, \Pi / I_w) \\
& \cong D \Ext_{\Pi}^2(\Pi / I_w, F) && \text{by the $2$-CY property} \\
& \cong D \Ext_{\Pi}^1(I_w, F) && \text{since } I_w = \Omega (\Pi / I_w) \\
& = 0.
\end{align*}

Now let $X \in \C(I_w)$. That is, $X$ is a $kQ$-split subquotient of (a finite sum of copies of) $I_w$. That is, $X$ is a $kQ$-split submodule of some $kQ$-split quotient $F$ of $I_w^n$.  By the discussion above we know that $\underline{\Ext}_{\Pi}^1(F, I_w)=0$. Now, using the right exactness of $\underline{\Ext}_{\Pi}^1$, we obtain $\underline{\Ext}_{\Pi}^1(X, I_w) = 0$.
\end{proof}

\begin{lemma}
For any $w \in W$, the category $\uperp I_w$ is closed under taking factors modulo finite dimensional modules.
\end{lemma}

\begin{proof}
Let $0\to A \to X \to F \to 0$ be a short exact sequence of $\Pi$-modules, such that $A$ is finite dimensional and $X \in \uperp I_w$. Then $\Hom_{\Pi}(A, I_w) = 0$, and so we have a monomorphism $\Ext_{\Pi}^1(F, I_w) \mono \Ext_{\Pi}^1(X, I_w)$. Note that the pullback of a $kQ$-split short exact sequence is $kQ$-split, so we obtain a commutative square
\[ \begin{tikzpicture}[xscale=3,yscale=-1.5]
 \node (A) at (0,0) {$\underline{\Ext}_{\Pi}^1(F, I_w)$};
 \node (B) at (1,0) {$\Ext_{\Pi}^1(F, I_w)$};
 \node (C) at (0,1) {$\underline{\Ext}_{\Pi}^1(X, I_w)$};
 \node (D) at (1,1) {$\Ext_{\Pi}^1(X, I_w)$};
 \draw [right hook->] (A) -- (B);
 \draw [right hook->] (C) -- (D);
 \draw [->] (A) -- (C);
 \draw [->] (B) -- (D);
\end{tikzpicture} \]
Since the right map is injective, so is the left one. Thus, since $X \in \uperp I_w$, also $F \in \uperp I_w$.
\end{proof}

\begin{corollary} \label{corollary.modfac_in_perp}
For any $w \in W$ we have $\Fac\C(I_w) \subseteq \uperp I_w$.
\end{corollary}

\begin{corollary} \label{cor.ext_filt_0}
For any $w \in W$ we have $\filt(\Fac\C(I_w)) \subseteq \uperp I_w$.
\end{corollary}

%\begin{proof}
%By Lemma~\ref{lemma.exact_1} we know that for $X, Y \in \mod \Pi$ there is an epimorphism from $\Hom_{kQ}(X, \tau Y)$ to $\underline{\Ext}_{\Pi}^1(X, Y)$. It follows that in any finite dimensional indecomposable $\Pi$-module, the preprojective $kQ$-summands sit above the non-preprojective $kQ$-summands. Moreover, the preprojective $kQ$-summands occur (from top down) in the order they occur in the AR-quiver of $kQ$ (from left to right) -- that is, the only $kQ$-summands sitting properly above a preprojective one, are $kQ$-modules which are properly to the left of this module in the AR-quiver of the preprojective component.
%
%Now let $F \in \filt(\Fac \mathscr{C}(I_w))$. Then, by the above observation, there is a $kQ$-split short exact sequence
%\[ 0\to F_{\text{non-preproj}} \to F \to F_{\text{preproj}}\to 0, \]
%with $\C(F_{\text{non-preproj}})$ not containing any preprojective objects, and $F_{\text{preproj}}$ filtered, in a $kQ$-split way, by preprojective $kQ$-modules in $\Fac \C(I_w)$. Now $F_{\text{non-preproj}} \in \uperp I_w$ since $\Hom_{kQ}(F_{\text{non-preproj}}, \tau I_w) = 0$ because $\C(I_w)$ and $\C(\tau I_w)$ consist of 
%preprojective $kQ$-modules. 
%Hence the claim follows from Corollary~\ref{corollary.modfac_in_perp} and the fact that $\uperp I_w$ is closed under $kQ$-split extensions.
%\end{proof}

We say that a short exact sequence of $\Pi$-modules:

$$0 \to Z \to A \to M \to 0,$$
with $M$ in some subcategory $\mathcal X$ of $\mod \Pi$, is a universal 
extension of $Z$ by $\mathcal X$,  if the induced map $M \to Z[1]$ is a right
$\mathcal X$-approximation.  
%
%, for any exact sequence 
%$ 0 \to Z \to^g B \to N \to 0$ with $N\in \mathcal X$, there is an exact
%commutative diagram:
%
%\[ \begin{tikzpicture}[xscale=2,yscale=-1.5]
%\node (A) at (0.5,0) {$0$};
%\node (B) at (1,0) {$Z$};
%\node (C) at (2,0) {$B$};
%\node (D) at (3,0) {$N$};
%\node (E) at (3.5,0) {$0$};
%\node (F) at (0.5,1) {$0$};
%\node (G) at (1,1) {$Z$};
%\node (H) at (2,1) {$A$};
%\node (I) at (3,1) {$M$};
%\node (J) at (3.5,1) {$0$};
%\draw [double distance=1.5pt] (B) -- (G);
%\draw [->] (C) -- (H);
%\draw [->] (D) -- (I);
%\draw [->] (A) -- (B);
%\draw [->] (B) -- node [above] {$g$} (C);
%\draw [->] (C) -- (D);
%\draw [->] (D) -- (E);
%\draw [->] (F) -- (G);
%\draw [->] (G) -- node [above] {$f$} (H);
%\draw [->] (H) -- (I);
%\draw [->] (I) -- (J);
%\end{tikzpicture} \]
%
The extension is called \emph{minimal} if the map is right minimal.
It follows directly that a minimal universal extension is unique up to
isomorphism.  We then have the following consequence of Corollaries~\ref{corollary.modfac_in_perp} and~\ref{cor.ext_filt_0}

\begin{corollary} \label{corollary.is_univ_extension}
For any $w \in W$, and $n$ such that $I_{c^n} \subseteq I_w$, the sequence
\begin{equation}\label{extn} 0\to I_{c^n} \to I_w \to M \to 0\end{equation}
is a minimal universal $kQ$-split extension of $I_{c^n}$ by objects in any of 
$\filt(\C(I_w))$, $\filt(\Fac \C(I_w))$, or $\uperp I_w$.
\end{corollary}

\begin{proof}
Since the sequence $0\to I_{c^n} \to \Pi \to \Pi / I_{c^n}\to 0$ is clearly $kQ$-split, so is the sequence $0\to I_{c^n} \to I_w \to M\to 0$. Since $M \in \filt(\C(I_w)) \subseteq \filt(\Fac \C(I_w)) \subseteq \uperp I_w$ by Corollary~\ref{cor.ext_filt_0}, the sequence is (\ref{extn}) is a $kQ$-split extension of $I_{c^n}$ by objects in any of $\filt(\C(I_w))$, $\filt(\Fac \C(I_w))$, or $\uperp I_w$. Moreover, since by definition $\underline{\Ext}_{\Pi}^1(\uperp I_w, I_w) = 0$, the extension is universal.

To see that it is minimal, consider the associated
map $M\to I_{c^n}[1]$.  
Since $I_{w}$ has no nonzero finite dimensional
summand, there is also no non-zero summand of $M$ which splits off, and
hence the map is
right minimal.  
\end{proof}

\begin{theorem}\label{four.seven}
Let $v, w \in W$. Let $Q$ be non-Dynkin, and assume that at least one of the following
conditions holds:
\begin{itemize}
\item[(i)] $\C(I_v)\subseteq \C(I_w)$,
\item[(ii)] $\Fac \C(I_v)\subseteq \Fac \C(I_w)$,
\item[(iii)] $\uperp I_v \subseteq \uperp I_w$,
\end{itemize}
then it follows that $I_v \subseteq I_w$.  
\end{theorem}

\begin{proof}
It has already been remarked that condition (i) is equivalent to 
$\filt (\C(I_v)) \subseteq \filt (\C(I_w))$, and similarly that 
(ii) is equivalent to $\filt (\Fac \C(I_v)) \subseteq 
\filt (\Fac \C(I_w))$.  We may therefore replace 
(i) and (ii) by these conditions.  

We consider the exact sequences:
\[ \begin{tikzpicture}[xscale=2.5,yscale=-1.5]
\node (X) at (-1,0) {$0$};
\node (Y) at (-1,1) {$0$}; 
\node (A) at (0,0) {$I_{c^n}$};
 \node (B) at (1,0) {$I_v$};
 \node (C) at (2,0) {$I_v / I_{c^n}$};
 \node (D) at (0,1) {$I_{c^n}$};
 \node (E) at (1,1) {$I_w$};
 \node (F) at (2,1) {$I_w / I_{c^n}$};
\node (Z) at (3,0) {$0$};
\node (W) at (3,1) {$0$};
 \draw [->] (A) -- node [above] {$\alpha$} (B);
 \draw [->] (B) -- (C);
 \draw [->] (D) -- node [above] {$\beta$} (E);
 \draw [->] (E) -- (F);
 \draw [double distance=1.5pt] (A) -- (D);
 \draw [dashed,->] (B) -- node [right] {$\varphi$} (E);
 \draw [dashed,->] (C) -- (F);
\draw [->] (X) -- (A);
\draw [->] (C) -- (Z);
\draw [->] (Y) -- (D);
\draw [->] (F) -- (W);
\end{tikzpicture} \]
By Corollary~\ref{corollary.is_univ_extension} they are universal extensions, and by assumption the sets we are universally extending by are contained one 
in the other. Hence we obtain the factorization indicated in the diagram.

Let $x\in W$.  Consider the short exact sequence
\[ 0\to[1.5] I_x \to[2]^{\iota_x} \Pi \to[1.5] \Pi / I_x\to[1.5] 0. \]
Since $\Pi / I_x$ is finite dimensional, we have $\Ext_{\Pi}^i(\Pi / I_x, \Pi) = 0$ for $i \in \{0,1\}$. Thus the above sequence induces an isomorphism
\[ \Hom_{\Pi}(I_x, \Pi) \cong \Hom_{\Pi}(\Pi, \Pi). \]
In other words, any map $I_x \to \Pi$ factors uniquely through the embedding $\iota_x$.

We then observe that $\iota_w \varphi \in \Hom_{\Pi}(I_v, \Pi)$, and hence $\iota_w \varphi = \lambda \iota_v$ for some $\lambda \in \Pi$. By the commutativity of the left square above we have
\[ \lambda \iota_{c^n} = \lambda \iota_v \alpha = \iota_w \varphi \alpha = \iota_w \beta = \iota_{c^n}, \]
so $\lambda = 1$. Hence we have a commutative triangle
\[ \begin{tikzpicture}[xscale=2.5,yscale=-1.5]
 \node (A) at (0,0) {$I_v$};
 \node (B) at (0,1) {$I_w$};
 \node (C) at (1,0.5) {$\Pi$};
 \draw [->] (A) -- node [left] {$\varphi$} (B);
 \draw [right hook->] (A) -- node [above] {$\iota_v$} (C);
 \draw [right hook->] (B) -- node [below] {$\iota_w$} (C);
\end{tikzpicture} \]
and thus $\varphi$ is an inclusion of ideals of $\Pi$.
\end{proof}

We  state here a lemma, essentially in \cite{BIRS}, 
which we will need to refer to in the sequel.  As usual, for $w\in W$, 
we denote by $\ell(w)$ the length of a reduced expression for 
$w$.  

\begin{lemma}\label{birs.lemma}
Suppose that $Q$ is a non-Dynkin quiver.  
Let $w$ be an element of the
Weyl group $W$.  
Assume that 
$\ell(s_iw)>\ell(w)$.  Then
we have the following:
\begin{enumerate}
\item $I_{s_iw}$ is properly contained in $I_w$.
\item $\Tor_1^\Pi(I_w,S_i)=0$.
\item The natural map $I_w\otimes I_i
\to I_{s_iw}$ is an isomorphism.
\end{enumerate}
\end{lemma}

\begin{proof}
(1) is part of \cite[Proposition III.1.10]{BIRS}.
%Since $\ell(s_iw)>\ell(w)$, \cite[Proposition III.1.9]{BIRS} establishes
%that $I_{s_iw}\ne I_w$, so $I_{s_iw}$ must be properly contained in $I_w$,
%establishing (1).  

Consider the short exact sequence $0\to I_i \to \Pi \to S_i\to 0$. Tensoring with $I_w$ we obtain the exact sequence
\[ 0\to \Tor^{\Pi}_1(I_w, S_i) \to I_{w} \otimes_{\Pi} I_i \to I_{w} \to I_{w} \otimes_{\Pi} S_i\to 0. \]

From \cite[Proposition III.1.1]{BIRS},
we know that at least one of 
$\Tor^\Pi_1(I_w,S_i)$ and $I_w \otimes_{\Pi} S_i$ is zero.  
The image of $I_w\otimes I_i$ inside $I_w$ is $I_{s_iw}$,
which is properly contained in $I_w$.  Thus $I_w \otimes_\Pi S_i$ is
non-zero, and it follows that $\Tor_1^\Pi(I_w,S_i)$ is zero.  This establishes
(2), and (3) also follows.  
\end{proof}

\section{Describing $\Pi/I_w$}\label{sectionfive}

In this section, we show that $I_w$ can be constructed from each of the
categories $\C(\Pi/I_w)$ and $\C(\Sub \Pi/I_w)$. This will be done by investigating the annihilators of the categories $\filt(\C(\Pi/I_w))$ and $\filt(\C(\Sub \Pi/I_w))$. The results of this section hold for arbitrary quivers, but 
will be applied only in the Dynkin case in the following section.

\begin{lemma} \label{lemma.Q}
The category
\[ \mathcal{Q} = \{ X \in \mod \Pi \mid X \text{ is a } kQ \text{-split quotient of an object in } \Sub \Pi / I_w \} \]
is closed under $kQ$-split extensions.
\end{lemma}

\begin{proof}
Let
\[ 0 \to X \to Y \to Z \to 0 \]
be a $kQ$-split short exact sequence, with $X$ and $Z$ in $\mathcal{Q}$. Then there are $kQ$-split epimorphisms $X' \epi X$ and $Z' \epi Z$ with $X'$ and $Z'$ in $\Sub \Pi / I_w$.

First consider the pullback along $Z' \epi Z$ as indicated in the following diagram.
\[ \begin{tikzpicture}[xscale=2,yscale=-1.5]
\node (X) at (0,0) {$0$};
\node (Y) at (1,0) {$X$};
\node (Z) at (2,0) {$Y'$};
\node (W) at (3,0) {$Z'$}; 
\node (A) at (4,0) {$0$};
\node (B) at (0,1) {$0$};
\node (C) at (1,1) {$X$};
\node (D) at (2,1) {$Y$};
\node (E) at (3,1) {$Z$};
\node (F) at (4,1) {$0$};
\draw [->] (B) -- (C);
\draw [->] (C) -- (D);
\draw [->] (D) -- (E);
\draw [->] (E) -- (F);
\draw [->] (X) -- (Y);
\draw [->] (Y) -- node [above] {$f$} (Z);
\draw [->] (Z) -- (W);
\draw [->] (W) -- (A);
\draw [double distance=1.5pt] (Y) -- (C);
\draw [->] (Z) -- (D);
\draw [->>] (W) -- (E);
\end{tikzpicture} \]
The map $Y' \to Y$ is a $kQ$-split epimorphism, since it is a pullback of the $kQ$-split epimorphism $Z' \epi Z$.

Now note that, if we denote by $K$ the kernel of the map $X' \epi X$, then right exactness of $\underline{\Ext}_{\Pi}^1$ from Proposition \ref{three.two}
implies that we obtain the following pullback diagram, and moreover 
that the lower short exact sequence is also $kQ$-split.
\[ \begin{tikzpicture}[xscale=2,yscale=-1.5]
\node (X) at (0,0) {$0$};
\node (Y) at (1,0) {$K$};
\node (Z) at (2,0) {$X'$};
\node (W) at (3,0) {$X$}; 
\node (A) at (4,0) {$0$};
\node (B) at (0,1) {$0$};
\node (C) at (1,1) {$K$};
\node (D) at (2,1) {$Y''$};
\node (E) at (3,1) {$Y'$};
\node (F) at (4,1) {$0$};
\draw [->] (B) -- (C);
\draw [->] (C) -- (D);
\draw [->] (D) -- (E);
\draw [->] (E) -- (F);
\draw [->] (X) -- (Y);
\draw [->] (Y) -- (Z);
\draw [->] (Z) -- (W);
\draw [->] (W) -- (A);
\draw [double distance=1.5pt] (Y) -- (C);
\draw [->] (Z) -- node [right] {$g$} (D);
\draw [->] (W) -- node [right] {$f$} (E);
\end{tikzpicture} \]
The map $f$ is a monomorphism with cokernel $Z'$ by the first diagram. It follows from the second diagram that $g$ is also a monomorphism with cokernel $Z'$. Thus $Y''$ is an extension of $Z'$ by $X'$.

Finally, by \cite[Proposition III.2.3]{BIRS}, the category $\Sub \Pi / I_w$ is extension closed, so $Y'' \in \Sub \Pi / I_w$, and hence $Y \in \mathcal{Q}$.
\end{proof}

\begin{proposition} \label{prop.I=Ann}
For $w \in W$ we have
\begin{align*}
I_w & = \Ann (\filt(\mathscr{C}(\Pi/I_w))) \\
& = \Ann (\filt(\mathscr{C}(\Sub \Pi/I_w))).
\end{align*}
In particular $\mathscr{C}(\Pi/I_w) \subseteq \mathscr{C}(\Pi/I_v)$ implies $I_v \subseteq I_w$ for any $v, w \in W$.
\end{proposition}

\begin{proof}
Since $\Pi / I_w \in \filt( \mathscr{C}( \Pi / I_w)) \subseteq \filt( \mathscr{C}( \Sub \Pi / I_w))$, and $\Ann \Pi/I_w = I_w$, we clearly have
\[ I_w \supseteq \Ann (\filt(\mathscr{C}(\Pi/I_w))) \supseteq \Ann (\filt(\mathscr{C}(\Sub \Pi/I_w))). \]
Thus it only remains to see that $I_w$ annihilates $\filt(\mathscr{C}(\Sub \Pi/I_w))$.

Now note that $\mathscr{C}(\Sub \Pi/I_w))$ is contained in the set $\mathcal{Q}$ of Lemma~\ref{lemma.Q} above. Since this set is closed under $kQ$-split extensions by Lemma~\ref{lemma.Q} it follows that also $\filt(\mathscr{C}(\Sub \Pi/I_w)) \subseteq \mathcal{Q}$. Thus it suffices to see that $I_w$ annihilates $\mathcal{Q}$. This however is clear, since the objects in $\mathcal{Q}$ are subquotients of $\add \Pi / I_w$ by definition.
\end{proof}

We also record here a lemma which we will need later:

\begin{lemma} \label{subclosed}
$\C(\Sub \Pi/I_w)$ is subclosed.
\end{lemma}

\begin{proof}
Let $M$ be a submodule of some module in $\C(\Sub \Pi/I_w)$. That means that $M$ is a submodule of $(\Pi / I_w)_{kQ}^n$ for some $n$. Note that $\Pi / I_w$ is a graded $\Pi$-module, and $(\Pi / I_w)_{kQ}$ is just the sum of the graded pieces $(\Pi / I_w)_d$. Thus $M$ is a submodule of $\oplus_{d=0}^{\infty} (\Pi / I_w)_d^{n_d}$, for suitable $n_d$. It follows that in the upper line of the following diagram, $M$ is embedded into the degree $0$ part of the $\Pi$-module on the right, where we
have written $(d)$ to indicate a shift of the grading by $d$.  
\[ \begin{tikzpicture}
 \node (A) at (0,0) {$M$};
 \node (B) at (2.5,0) {$\bigoplus_{d=0}^\infty (\Pi/I_w)_d^{n_d}$};
 \node (D) at (6.5,0) {$\bigoplus_{d=0}^{\infty} (\Pi / I_w)(d)^{n_d}$};
 \node (C) at (0,-1.5) {$M \otimes_{kQ} \Pi$};
 \draw [right hook->] (A) -- (B);
 \draw [right hook->] (B) -- (D);
 \draw [right hook->] (A) -- (C);
 \draw [dashed,->] (C) -- (D);
\end{tikzpicture} \]
By Hom-tensor adjointness we obtain a degree-preserving $\Pi$-linear map 
as indicated by the dashed arrow above. In particular its image $Y$ is a graded $\Pi$-submodule of $\bigoplus_{d=0}^{\infty} (\Pi / I_w)(d)^{n_d}$. Looking at degree $0$, we see that $Y_0 \cong M$, and clearly $Y_0 \in \C(Y) \subseteq \C(\Sub \Pi/I_w)$.
\end{proof}

\section{Connection between the ideals $I_w$ and the quotients $\Pi/I_w$ in the Dynkin case}\label{sectionsix}
In Section~\ref{sectionfour} we have seen that $\C(I_v)\subseteq \C(I_w)$ 
implies $I_v \subseteq I_w$ for $Q$ non-Dynkin.  In order to prove the
same result
in the Dynkin case as well, 
we prove in this section that each $I_w$ is dual to some
$\Pi/I_{w'}$.  
This allows us to work with $\Pi/I_{w'}$ instead of 
$I_w$, so that we can use the results from Section~\ref{sectionfive} 
describing $\Pi/I_{w'}$ to achieve the desired result.

Throughout this section 
let $\Pi$ be the preprojective algebra of a Dynkin quiver. 
We write $\mod \Pi^\op$ for the category of right $\Pi^\op$-modules (or 
equivalently left $\Pi$-modules).  

The following lemma is the Dynkin analogue of Lemma~\ref{birs.lemma}.  
For $w \in W$ we denote as usual by $\ell(w)$ the length of a shortest expression for $w$.

\begin{lemma}\label{six.zero}
Let $Q$ be Dynkin, and let $w$ be an element of the Weyl group 
$W$.
Assume that $\ell(s_iw)> \ell(w)$.  Then we have
the following:
\begin{enumerate}
\item 
$I_{s_iw}$ is properly contained in $I_w$.
\item 
$\Tor_1^\Pi(I_w,S_i)=0$.
\item The natural map $I_w \otimes I_i \to I_{s_iw}$
is an isomorphism.
\end{enumerate}
\end{lemma}

\begin{proof}
Let $\widehat Q$ be a 
non-Dynkin quiver containing $Q$ as a full subquiver.
Let $\Pi$ be the preprojective algebra for $Q$, and
$\widehat \Pi$ the preprojective algebra
for $\widehat Q$. Denote by $\widehat I_w$
and $\widehat I_i$ the corresponding ideals in
$\widehat \Pi$.  
We have the exact sequence:
$$\Tor_1^{\widehat \Pi}(\widehat I_w,S_i) \to 
\widehat I_w\otimes \widehat I_i \to 
\widehat I_w \to
\widehat I_w\otimes S_i.$$
$\widehat I_{s_iw}$ is the image of 
$\widehat I_w\otimes \widehat I_i$ inside $\widehat I_w$.  
By Lemma~\ref{birs.lemma}, we know
$\Tor_1^{\widehat \Pi}(\widehat I_w,S_i)=0$, and that 
$\widehat I_{s_iw}$ is properly contained in $\widehat I_w$.   

(1) Since $\widehat I_{s_iw}$ is properly contained in 
$\widehat I_w$, we have a proper epimorphism $\widehat \Pi
/\widehat I_{s_iw} \to \widehat \Pi/\widehat I_w$.  Since
$\widehat \Pi/\widehat I_{s_iw} \cong \Pi/I_{s_iw}$ and 
$\widehat \Pi/\widehat I_w \cong \Pi/I_w$, we have a proper epimorphism
$\Pi/I_{s_iw}\to \Pi/I_w$.  Hence we have that 
$I_{s_iw}$ is properly contained in $I_w$.

(2) As discussed above, $\Tor^{\widehat\Pi}_1
(\widehat I_w,S_i)=0$.  Further, 
\begin{eqnarray*}\Tor_1^{\widehat \Pi}(\widehat I_w,S_i)
&\cong& \Tor_2^{\widehat \Pi} 
(\widehat \Pi/\widehat I_w,S_i)
\cong \Tor_2^{\widehat \Pi}
(\Pi/I_w,S_i)\cong 
D(\Ext^2_{\widehat \Pi}(\Pi/I_w,S_i))\\
&\cong& 
\Hom_\Pi(S_i,\Pi/I_w) \cong 
\underline\Hom_\Pi(S_i,\Pi/I_w),\end{eqnarray*}
by using \cite{ce} and the 2-CY 
property for finite length
$\widehat \Pi$-modules.

On the other hand, we have  $$\Tor_1^\Pi
(I_w,S_i)\cong \Tor_2(\Pi/I_w,S_i)
\cong D(\Ext^2_\Pi(\Pi/I_w,S_i))
\cong \underline\Hom
(S_i,\Pi/I_w),$$ using \cite{ce} and
the 2-CY property for $\underline 
\mod(\Pi)$.  

Since $\Tor_1^{\widehat \Pi}(\widehat I_w,S_i)
=0$, we then have 
$\Hom_\Pi(S_i,\Pi/I_w)=0$, hence 
$\underline\Hom_\Pi(S_i,\Pi/I_w)=0$, and
so 
$\Tor_1^\Pi(I_w,S_i)=0$.  

(3) Consider the exact sequence 
$$\Tor_1^\Pi(I_w,S_i) \to I_w
\otimes_\Pi I_i \to 
I_{w} \to I_w \otimes_\Pi S_i.$$  
By definition, $I_{s_iw}$ is the image of $I_w\otimes I_i$ in $I_w$.   
The result now
follows from (2). 
\end{proof}

\begin{lemma} \label{lemma.ext_vanish}
Let $w \in W$.
\begin{enumerate}
\item If $\ell(s_i w) > \ell(w)$ then $\Ext_{\Pi}^1(I_{w}, S_i) = 0 = \Ext_{\Pi}^1(S_i, I_{w})$.
\item If $\ell(w s_i) > \ell(w)$ then $\Ext_{\Pi^\op}^1(I_{w}, S_i) = 0 = \Ext_{\Pi^\op}^1(S_i, I_{w})$.
\end{enumerate}
\end{lemma}

\begin{proof}
We only prove (1), since (2) is dual.

Consider the exact sequence
\[ 0\to \Tor^{\Pi}_1(I_w, S_i) \to I_{w} \otimes_{\Pi} I_i \to I_{w} \to I_{w} \otimes_{\Pi} S_i\to 0. \]
We know that if $\ell(s_i w) > \ell(w)$ then
%then $I_{w} \otimes_{\Pi} I_i = I_{s_i w} \embed I_{w}$, hence 
$\Tor^{\Pi}_1(I_{w}, S_i) = 0$ by Lemma~\ref{six.zero}.

Now note that $\Tor^{\Pi}_1(I_{w}, S_i) = D \Ext_{\Pi}^1(I_{w}, \underbrace{DS_i}_{= S_i})$ \cite{ce}.

Hence the claim that $\Ext_{\Pi}^1(S_i, I_{w}) = 0$ follows from the $2$-CY property of the stable category $\underline{\mod} \Pi$.  
\end{proof}

\begin{lemma} \label{lemma.I_makes_equal}
\begin{enumerate}
\item Let $M \subseteq N$ in $\mod \Pi^\op$. Then $I_i M = I_i N$ if and only if $N / M \cong S_i^n$ for some $n$.
\item Let $M \subseteq N$ in $\mod \Pi$. Then $M I_i = N I_i$ if and only if $N / M \cong S_i^n$ for some $n$.
\end{enumerate}
\end{lemma}

\begin{proof}
We only prove (1), since (2) is dual.

Consider $0\to M \to N \to N/M\to 0$ in $\mod \Pi^\op$. Multiplying with $I_i$ we obtain the complex
\[ 0 \to I_i M \to I_i N \to I_i N/M \to 0 \]
whose homology is concentrated in the middle.

If we assume $I_i M = I_i N$ then it follows that $I_i (N/M) = 0$, and hence $N/M \cong S_i^n$ for some $n$.

Assume now conversely that $N/M \cong S_i^n$ for some $n$. Then $I_i \otimes_{\Pi} N/M = 0$, since $I_i^2=I_i$, and hence the map $I_i \otimes_{\Pi} M \to I_i \otimes_{\Pi} N$ is onto. It follows that the inclusion map $I_i M \embed I_i N$ is also onto, and hence $I_i M = I_i N$.
\end{proof}

We write $w_0$ for the longest element in $W$ (which only exists when
$W$ is finite).

\begin{proposition} \label{prop.dualities}
Let $w \in W$.
\begin{enumerate}
\item $D I_w \cong \Pi / I_{w_0 w^{-1}}$ in $\mod \Pi^\op$.
\item $D I_w \cong \Pi / I_{w^{-1} w_0}$ in $\mod \Pi$.
\end{enumerate}
\end{proposition}

\begin{proof}
We only prove (1), since (2) is dual.

Since $\Pi$ is self-injective, we have 
$\Pi \cong D \Pi$ in $\mod \Pi^\op$. Hence we have an epimorphism $\Pi \epi D I_w$ of $\Pi^\op$-modules. Let $\widetilde{I}_w$ be its kernel. (So $\widetilde{I}_w$ is a left ideal.) We want to show that $\widetilde I_w=I_{w_0w^{-1}}$.

Since $I_w \cdot I_{w_0 w^{-1}} = 0$, and hence also $I_{w_0 w^{-1}} \cdot D I_w = 0$, we have $I_{w_0 w^{-1}} \subseteq \widetilde{I}_w$.
We now show by induction on $\ell(w)$ that $I_{w_0 w^{-1}} = \widetilde{I}_w$.
For $w = e$ we have $I_{w_0 w^{-1}} = 0 = \widetilde{I}_w$.

Assume the claim holds for $w$, so that $\widetilde I_w=I_{w_0w^{-1}}$ 
and assume $\ell(s_i w) > \ell(w)$. Then we have an inclusion $I_{s_i w} = I_w I_{s_i} \embed I_w$ of $\Pi$-modules. By Lemma~\ref{lemma.I_makes_equal}(2) we have that $I_w / I_{s_i w} \cong S_i^n$ for some $n$ as $\Pi$-modules. Dualizing we obtain the following short exact sequence in $\mod \Pi^\op$:
\[ 0\to S_i^n \to D I_w \to D I_{s_i w}\to0 \]
Then
we obtain the following diagram in $\mod \Pi^\op$.
\[ \begin{tikzpicture}[xscale=2,yscale=-1.5]
 \node (X) at (-1,0) {$0$};
 \node (Y) at (-1,1) {$0$};
 \node (A) at (0,0) {$\widetilde{I}_w$};
 \node (B) at (1,0) {$\Pi$};
 \node (C) at (2,0) {$D I_w$};
 \node (D) at (0,1) {$\widetilde{I}_{s_i w}$};
 \node (E) at (1,1) {$\Pi$};
 \node (F) at (2,1) {$D I_{s_i w}$};
 \node (G) at (2,-1) {$S_i^n$};
 \node (W) at (3,0) {$0$};
 \node (Z) at (3,1) {$0$};
\draw [->] (A) -- (B);
 \draw [->] (B) -- (C);
 \draw [->] (D) -- (E);
 \draw [->] (E) -- (F);
 \draw [->] (A) -- (D);
 \draw [double distance=1.5pt] (B) -- (E);
 \draw [right hook->] (G) -- (C);
 \draw [->>] (C) -- (F);
 \draw [->] (X) -- (A);
 \draw [->] (Y) -- (D);
 \draw [->] (C) -- (W);
 \draw [->] (F) -- (Z);
\end{tikzpicture} \]
By the snake lemma the left vertical map is an inclusion with cokernel $S_i^n$. Thus by Lemma~\ref{lemma.I_makes_equal}(1) we have $I_i \widetilde{I}_{s_i w} \subseteq \widetilde{I}_w$. By our induction assumption, we have $\widetilde{I}_w = I_{w_0 w^{-1}}$. Since
\[ \ell(w_0 w^{-1} s_i) = \ell(w_0) - \ell(w^{-1} s_i) = \ell(w_0) - \ell(s_i w) < \ell(w_0) - \ell(w) = \ell(w_0 w^{-1}) \]
we have $I_i I_{w_0 w^{-1} s_i} = I_{w_0 w^{-1}}$. Putting all these together we have $I_i \widetilde{I}_{s_i w} \subseteq I_i I_{w_0 w^{-1} s_i}$, 
and clearly $I_iI_{w_0w^{-1}s_i}\subseteq I_i \widetilde{I}_{s_iw}$.
So, by Lemma~\ref{lemma.I_makes_equal}(1) the quotient $\widetilde{I}_{s_i w} / I_{w_0 w^{-1} s_i}$ is isomorphic to $S_i^m$ for some $m$ as $\Pi^\op$-modules.

Now note that by Lemma~\ref{lemma.ext_vanish}(2) we have $\Ext_{\Pi^\op}^1(S_i, I_{w_0 w^{-1} s_i}) = 0$. Hence $\widetilde{I}_{s_i w}$ is the direct sum of $I_{w_0 w^{-1} s_i}$ and a left ideal $R$ of $\Pi$ which is isomorphic to $S_i^m$. If $R = 0$, then we are done, so assume $R \neq 0$. The only non-zero left ideal of $\Pi$ which is isomorphic to $S_i^m$ for some $m$ is $I_{w_0 s_i}$. By assumption $R \cap I_{w_0 w^{-1} s_i} = 0$, so $I_{w_0 s_i} \not\subseteq I_{w_0 w^{-1} s_i}$. But $I_{w_0 s_i} = I_{w_0 w w_0 w_0 w^{-1} s_i} = I_{w_0 w^{-1} s_i} I_{w_0 w w_0}$, since $\ell(w_0 w^{-1} s_i) + \ell(w_0 w w_0) = [\ell(w_0) - \ell(w) - 1] + \ell(w) = \ell(w_0) - 1 = \ell(w_0 s_i)$, and hence $I_{w_0 w^{-1} s_i} \supseteq I_{w_0 s_i}$. This is a contradiction. So $R = 0$, and hence $\widetilde{I}_{s_i w} = I_{w_0 w^{-1} s_i}$.
This finishes the induction step, and hence the proof of the lemma.
\end{proof}

For $v,w$ in $W$, it is said that 
$v$ is less than $w$ in the Bruhat order, written $v\leq w$, if, fixing
a reduced word for $w$, it is possible to find a subexpression of
that word which equals $v$.  

\begin{lemma}\label{Dynkin.containment}
$I_v$ contains $I_w$ iff $v\leq w$ in the Bruhat order.
\end{lemma}

\begin{proof}
This is known in the extended Dynkin case \cite{IR}.
To see that it follows in the Dynkin case, denote by $\widehat \Pi$ the preprojective algebra of the corresponding Euclidean quiver. Observe that 
$\Pi=\widehat \Pi / \widehat I_{w_0}$.  Since $w_0$ is the 
maximal element of $W$ under Bruhat order, the ideal $\widehat I_{w_0}$ is
contained in $\widehat I_w$ for any $w\in W$.  Thus, 
$I_{i_1}\dots I_{i_r} = (\widehat I_{i_1}\dots \widehat I_{i_r})/\widehat I_{w_0}$,
so containment relations agree in the two algebras.  
\end{proof}

\begin{lemma} \label{notesDDynkin}
Let $Q$ be a Dynkin quiver, and $v, w \in W$. If $\C(I_v) \subseteq \C(I_w)$ then $I_v \subseteq I_w$.  
\end{lemma}

\begin{proof}
By Proposition~\ref{prop.dualities} we have that (as $\Pi$-modules) $I_v \cong D \Pi/I_{w_0 v^{-1}}$ and $I_w \cong D \Pi/I_{w_0 w^{-1}}$. Thus the assumption may be rewritten as
\[ \C(D \Pi/I_{w_0 v^{-1}}) \subseteq \C(D \Pi/I_{w_0 w^{-1}}). \]
Now note that dualizing commutes with restricting to the path algebra, so the above inclusion is equivalent to the inclusion
\[ \C_{kQ^{\op}}(\Pi/I_{w_0 v^{-1}}) \subseteq \C_{kQ^{\op}}(\Pi/I_{w_0 w^{-1}}). \]

By Proposition~\ref{prop.I=Ann} (for $kQ^{\op}$) this implies that $I_{w_0 w^{-1}} \subseteq I_{w_0 v^{-1}}$. By Lemma~\ref{Dynkin.containment} we deduce that
$w_0w^{-1}\geq w_0v^{-1}$ in Bruhat order.  
The map $u\mapsto u^{-1}$ is clearly an automorphism of Bruhat order,
while the map $u\mapsto
w_0u$ is an anti-automorphism of Bruhat order \cite[Proposition 2.3.4]{BB}, 
so we deduce that 
$v\geq w$.   Thus, by Lemma~\ref{Dynkin.containment} again, we deduce that
$I_v\subseteq I_w$.

\end{proof}

%We also note the following lemma, which we shall need subsequently.
%\comment{HT: actually, I'm not sure we do need it, but I only figured that
%out after I'd written out a proof, so I thought it might as well stay here
%until we were sure whether or not we needed it.  Note that I assume the
%ideals are isomorphic as bimodules, which is stronger than we might 
%wish for.}

%\begin{lemma} \label{Dynkin.iso} Let $v,w \in W$.  
%If $I_v$ and $I_w$ are isomorphic as $\Pi$--$\Pi$ bimodules, 
%then $v=w$ and $I_v=I_w$.  \end{lemma}

%\begin{proof}
%The proof is by downward induction on the length of the longer of $v,w$.
%The base case is when one of $v,w$, is the unique longest word of 
%$W$.  Let us assume it is $w$.  Then $I_w=0$.  Since $I_v\cong I_w$, 
%$I_v=0$ also.  By Lemma~\ref{Dynkin.containment}, it follows that
%$v=w$.  

%Now suppose that we have established the result provided that at least 
%one of $\ell(v)$, $\ell(w)$ is greater than $m$.  Suppose then that
%$\ell(v)\leq \ell(w)= m$.  Find $i$ such that $\ell(s_iw)>\ell(w)$.  
%By Lemma~\ref{six.zero}, we know that the natural map 
%$I_{w}\otimes I_i \rightarrow I_{s_iw}$ is an isomorphism.  
%It follows that 
%$I_v \otimes I_i \rightarrow I_{s_iv}$ is also an isomorphism,
%and that $I_{s_iw}\cong I_{s_iv}$.  Since $\ell(s_iv)>\ell(v)$, we can
%apply the induction hypothesis to deduce that $s_iw=s_iv$, and thus
%$w=v$, and $I_w=I_v$. 
%\end{proof}

\section{Proof of main results}

In this section, we prove our main results, stated in Section 2 as Theorems
\ref{two.two} and \ref{two.three}.   In particular, we establish 
that the cofinite
quotient closed subcategories of the category of preprojective
$kQ$-modules are exactly those of the form $\C(I_w)$ for some element
$w$ in the Weyl group $W$.

%\comment{Somewhere, perhaps here, we should have the discussion about why
%cofinite quotient closed preprojective is equivalent to cofinite quotient
%closed.}

\begin{lemma} \label{simpleproj}
For $w\in W$, and $i$ a source, we have that $\C( I_w)$ contains 
$S_i$, the simple projective at $i$, iff $w$ has no reduced expression
as $w=s_iv$.  
\end{lemma}

\begin{proof}
Observe first that a $\Pi$-module $M$ has $S_i\in \C(M)$ iff it has
$S_i$ in its top, since otherwise (because $i$ is a source), whatever is 
``sitting over'' $S_i$ will also be sitting over it as a $kQ$-module.  
So the problem reduces to asking when $I_w$ has $S_i$ in its top.  

%As in \cite[Proof of Proposition III.1.5(a)]{BIRS}, 
%we know that $I_i \otimes^L I_{s_iv}=I_i \otimes I_{s_iv}$.  We therefore
%get an exact sequence:
Consider the exact sequence:
$$\Tor_1(I_{w},S_i)\to I_{w}\otimes I_i \to I_{w}
\to I_{w}\otimes S_i \to 0.$$

Suppose first that $\ell(s_iw)>\ell(w)$.  The image of $I_w\otimes I_i$ 
in $I_w$ is $I_{s_iw}$, which is strictly contained in $I_w$, 
by Lemma~\ref{birs.lemma} or  Lemma~\ref{six.zero}.  Thus we get a nonzero quotient, which is
necessarily semisimple as a $\Pi$-module, since it is annihilated by
$I_i$.  Thus it is of the form $S_i^n$ for some $n>0$, 
and we deduce that $I_{w}$ has $S_i$ in its top.  

Now suppose that $\ell(s_iw)<\ell(w)$. We can write $w=s_iv$ as a reduced
product.   
The image of $I_{s_iv}\otimes I_i$ in $I_{s_iv}$ is  $I_{s_iv}I_i=I_{s_iv}$.
Thus the map $I_{s_iv}\otimes I_i \to I_{s_iv}$ is surjective.  It follows that $I_{s_iv}\otimes S_i$ is 
zero, and thus that $I_{s_iv}$ does not have $S_i$ in its top.  
\end{proof}

\begin{lemma} \label{difference}
Let $i$ be a source, so $S_i$ is a simple projective.  
Assume $\ell(s_iw)>\ell(w)$.  Then $\C(I_w)=\C(I_{s_iw}) \cup \{S_i\}$.
\end{lemma}

\begin{proof} We have already established that $\C(I_w)$ contains
$S_i$ while $\C(I_{s_iw})$ does not.  Observe that we have a short 
exact sequence:
$$ 0 \to I_{s_iw}\to I_w \to I_w\otimes S_i 
\to 0.$$
View this sequence as a sequence of $kQ$-modules.  
Since the righthand side is projective
as a $kQ$-module, the sequence splits.  Thus the summands of $(I_w)_{kQ}$ other
than those surviving in $(I_w\otimes S_i)_{kQ}$ coincide with  those 
in $(I_{s_iw})_{kQ}$.  
\end{proof}

Number the vertices of $Q$ from $1$ to $n$ so that if there is an arrow
$i \to j$ then $i<j$. 

\begin{theorem}\label{main}
Any cofinite, quotient closed subcategory $\mathcal A$ of the preprojective
$kQ$-modules appears as $\C(I_w)$ for some $w\in W$ such that $\ell(w)$ is
the number of missing preprojective modules.
% each $w\in W$ gives rise
%in this way to a cofinite, quotient closed subcategory. 

Such a $w$ can be found as described in Section 2.  
Number the vertices of $Q$ from $1$ to $n$, so that if there is an arrow
$i \to j$, then $i<j$.  
Order the indecomposable preprojective modules as 
$$P_1,\dots,P_n,\tau^{-1}P_1,\dots,\tau^{-1}P_n,\tau^{-2}P_1,\dots$$
Let $\mathcal X$ be 
the indecomposable modules missing from $\mathcal A$.  Take these
indecomposables in the induced order, and read $\tau^{-j}P_i$ as 
$s_i$.  The result is the leftmost word for $w$ in $\underline c^\infty$, where
$\tau^{-j}P_i$ is identified with the $j$-th instance of $s_i$ in
$\underline c^\infty$.  
\end{theorem}

Before we prove this theorem, we first state and prove two lemmas.

Let $\mathcal A$ be a cofinite, quotient closed subcategory
of the preprojective $kQ$-modules.  Let $S_1$ be the simple projective
$kQ$-module associated to the vertex 1, and let $P$ be the sum of the other
indecomposable projectives.  Define $T=P\oplus \tau^{-1}(S_1)$.  Let 
$Q'=\mu_1(Q)$.  The associated reflection functor is $R_1^+=\Hom_{kQ}(T, \ )$.  
If $Q$ is 
non-Dynkin, then let $\mathcal A'$ be the subcategory of 
$kQ'$-modules  given by $\Hom(T,\mathcal A)$.
If $Q$ is Dynkin, define 
$\mathcal A'$ to consist of the additive subcategory of 
$\mod kQ'$ generated by $\Hom(T,\mathcal A)$ together with $S_1'$, the
new simple $kQ'$-module.

\begin{lemma}\label{setup} Let $\mathcal A'$, $Q'$, $S_1'$ be as defined
above.  Then \begin{enumerate}\item $\mathcal A'$ is quotient closed, and
\item if we have
a short exact sequence of $kQ'$-modules
$$0\to Y' \to Z' \to (S'_1)^r\to 0,$$
with $Y'\in \mathcal A'$ and $Z'$ preprojective, then $Z'\in \mathcal A'$.
\end{enumerate}
 
\end{lemma}

\begin{proof}
(1) Let $X \in \mathcal A$, so $\Hom(T,X)\in \mathcal A'$.  Denote 
$\Hom(T,X)$ by $X'$, and suppose that there is an epimorphism from 
$X'$ to $Y'$, with $Y'$ a preprojective $kQ'$-module.  
We want to show that $Y'\in \mathcal A'$.  

Since $S_1'\in \mathcal A'$ if it is preprojective, by construction, 
we may assume that $Y'$ 
has no summands isomorphic to $S_1'$.  By assumption, we have a short exact sequence in
$kQ'$-mod:
$$0\to K' \to X' \to Y' \to 0.$$

Since $R_1^+$ is an equivalence of categories from the additive hull of 
the indecomposable objects $kQ$-mod other than $S_1$ to the additive hull
of the indecomposable objects of $kQ'$-mod other than $S_1'$, and our
short exact sequence lies in the latter subcategory, there is a 
corresponding short exact sequence in $kQ$-mod, which shows that 
there is an epimorphism from $X$ to $Y$, and thus that $Y\in \mathcal A$,
so $Y'\in \mathcal A'$, as desired.  

(2) We may assume that $Y'$ and $Z'$ have no summands isomorphic to $S'_1$, so we may write 
$Y'=\Hom(T,Y)$ with $Y\in \mathcal A$, and $Z'=\Hom(T,Z)$.  The given short exact sequence of 
$kQ'$-modules then implies the existence of a short exact sequence of 
$kQ$-modules $0\to S_1^r \to Y \to Z \to 0$.  Since $\mathcal A$ is closed
under surjections, $Z\in \mathcal A$, so $Z'\in \mathcal A'$.   
\end{proof}

\begin{lemma} \label{step}
Let $\mathcal A$ be a cofinite, quotient closed subcategory
of the preprojective $kQ$-modules.  Let $\mathcal A'$ be defined as above.  

%Let $S_1$ be the simple, projective 
%$kQ$-module associated to the vertex $1$, and suppose that $S_1$ is not
%in $\mathcal A$.  
%Let $P$ be the sum of the 
%other indecomposable projectives, and define $T=P\oplus \tau^{-1} S_1$.  
%Let $Q'=\mu_1(Q)$.  If $Q$ is 
%non-Dynkin, then let $\mathcal A'$ be the subcategory of 
%$kQ'$-modules  given by $\Hom(T,\mathcal A)$.
%If $Q$ is Dynkin, define 
%$\mathcal C'$ to consist of the additive subcategory of 
%$\mod kQ'$ generated by $\Hom(T,\mathcal A)$ together with $S_1'$, the
%new simple $kQ'$-module.  This category $\mathcal A'$ is also quotient-closed.

Suppose that Theorem \ref{main} holds for $\mathcal A'$.  Then it also holds for 
$\mathcal A$.  
\end{lemma}

\begin{proof}
%$\mathcal A'$ is quotient closed because if $\Hom(T,X )\epi\Hom(T,Y)$, and
%$S_i$ does not appear as a direct summand of $Y$, then $X\epi Y$.  Therefore
%the assumption that the theorem holds for $\mathcal A'$ is meaningful.  
The assumption that the theorem holds for $\mathcal A'$ tells us that 
$\mathcal A'=\C_{Q'}(I_w)$ for $w$ obtained by reading the AR quiver for
$kQ'$ (starting with $P_2,P_3,\dots$) and recording $s_i$ for each 
indecomposable object
which is not in $\mathcal A'$.  

We claim that 
$\ell(s_1w)>\ell(w)$.  
Seeking a contradiction, 
suppose that $w=s_1v$ is a reduced expression for $w$.  Thus 
$I_w=I_vI_{s_1}$.  We claim that $\C(I_v)\subseteq \C(I_w)$.  The reason is
that modules in $\C(I_v)$ are extensions of $S_1$ by
some object from $\C(I_w)$, but $\C(I_w)$ is closed under such extensions
by Lemma \ref{setup}(2).  So $\C(I_v)\subseteq \C(I_w)$.
At the same time, $I_v$ strictly contains $I_w$, by Lemma~\ref{birs.lemma} or 
Lemma~\ref{six.zero}.  These two statements together
contradict Theorem~\ref{four.seven} or Lemma~\ref{notesDDynkin}.  
Therefore we conclude that $\ell(s_1w)>\ell(w)$.  

Write $\mathcal A^+$ for the additive category generated by 
$\mathcal A$ together with
$S_1$, and write $\mathcal A^-$ for the additive category generated by
the indecomposables of $\mathcal A$ excluding $S_1$.  

We now have either $\mathcal A=\mathcal A^-$ or 
$\mathcal A=\mathcal A^+$.  Since $\mathcal A'= R^+_1(\mathcal A^+)
=R^+_1(\mathcal A^-)$, both $\mathcal A^+$ and $\mathcal A^-$ 
satisfy our hypotheses, so we need to prove that the theorem holds
for both of them.  We first treat $\mathcal A^-$.  

We recall that Theorem~\ref{thm.tensoringcommutes} says that the following
diagram commutes:

\[ \begin{tikzpicture}[xscale=5,yscale=-1.5]
 \node (A) at (0,0) {$\Mod \Pi$};
 \node (B) at (1,0) {$\Mod \Pi$};
 \node (C) at (0,1) {$\Mod k\widet{Q}$};
 \node (D) at (1,1) {$\Mod kQ$};
 \draw [->] (A) -- node [above] {$- \otimes_{\Pi} I_1$} (B);
 \draw [->] (A) -- node [left] {res} (C);
 \draw [->] (B) -- node [right] {res} (D);
 \draw [->] (C) -- node [above] {$- \otimes_{k\widet{Q}} T$} (D);
\end{tikzpicture} \]

If we start with 
$I_w$ in the upper lefthand corner, we get $\mathcal A'$ in the bottom
left, and thus $\mathcal A^-$ in the bottom right.  On the other
hand, in the upper right corner, we have 
$I_w\otimes I_{1}$.  Since $\ell(s_1w)>\ell(w)$, this
is $I_{s_1w}$ by Lemma~\ref{birs.lemma} or Lemma~\ref{six.zero}.  Therefore, $\C(I_{s_1w})=\mathcal A^-$.

Now we establish the link to the leftmost word for $s_1w$.  Since
$s_1w$ admits $s_1$ on the left,
the leftmost word for $s_1w$ begins with $s_1$ (corresponding in the
AR-quiver to the simple projective $S_1$).  The rest of the leftmost
word for $s_1w$ is the leftmost word for $w$ in the AR quiver for $Q'$,
and by assumption, this corresponds to the indecomposable objects not in
$\mathcal A'$.  

Now we consider $\mathcal A^+$.  Using the result which we have already
established for $\mathcal A^-$, Lemma \ref{difference} tells us that 
$\C(I_w)=\mathcal A^+$.  Since $w$ does not admit $s_1$ as a leftmost
factor, the leftmost word for $w$ in the AR quiver of $kQ$ is the
same as the leftmost word for $s_1w$ with the initial $s_1$ removed.  
This establishes the desired result for $\mathcal A^+$.   
\end{proof}

\begin{proof}[Proof of Theorem \ref{main}]
We establish the theorem by induction on $m$, the number of 
indecomposable objects missing from $\mathcal A$, and on $p$, the 
position of the first indecomposable missing from $\mathcal A$ 
in the order on the indecomposable objects of $\mathcal A$.  

The statement is clear if $\mathcal A$ has no missing indecomposables.  
(The prescription for finding $w$ gives us the
empty word, which is the unique reduced word for the identity element $e$, and $\C(I_e)$ is the whole preprojective component.)

Now, let $\mathcal A$ be some cofinite, quotient closed subcategory of 
the preprojective $kQ$-modules, with $m$ missing indecomposables, and
with the first missing indecomposable in position $p$.  
Suppose that we already know that the theorem holds for 
for any quotient closed subcategory with fewer than $m$ missing indecomposables, or with exactly $m$ missing indecomposables and with the first
missing indecomposable in a position earlier than $p$.  
Define $\mathcal A'$
as above.  

If $p=1$, then $\mathcal A$ does not contain $S_1$.  In this case, 
$\mathcal A'$ has fewer missing 
indecomposables than $\mathcal A$ does.
If $p>1$, then $\mathcal A$ and $\mathcal A'$ have the same number of 
missing indecomposables, but the first missing indecomposable is earlier
in $\mathcal A'$ than it is in $\mathcal A$.  

Thus, in either case, the induction hypothesis tells us that the statement
of the theorem holds for $\mathcal A'$, and Lemma \ref{step} tells us that
it also holds for $\mathcal A$.  
\end{proof}

We also have a converse.  

\begin{theorem}\label{converse}
$\C(I_w)$ is quotient closed for any $w\in W$.  
\end{theorem}

\begin{proof} 
If $Q$ is non-Dynkin, we proceed as follows.  
Consider $\Facpp \C(I_w)$, where we write $\Facpp \mathcal A$ for the 
part of $\Fac \mathcal A$ consisting of preprojective modules.  
$\Facpp \C(I_w)$  is quotient closed and clearly cofinite,
so by Theorem~\ref{main}, $\Facpp \C(I_w)=\C(I_v)$ for some $v$,
and hence $\Fac \C(I_w)=\Fac \C(I_v)$.  Thus $I_w=I_v$, by Theorem
\ref{four.seven}.  Thus $\C(I_w)$ is quotient closed.  
%However, we have seen in Section~\ref{sectionfour} that $I_v$ is 
%determined by $\C(I_v)$, and $I_w$ by $\Fac \C(I_w)$, hence $I_v \cong I_w$.
%It then follows from \cite{BIRS} that $I_v=I_w$, and that $v=w$.  

Now, in the Dynkin case, for $x\in W$,  
$\C(\Sub(\Pi/I_x))$ is subclosed by Lemma~\ref{subclosed}, so must equal 
$\C(\Pi/I_v)$ for some $v\in W$ (applying Theorem~\ref{main} after dualizing).
By Proposition \ref{prop.I=Ann}, $\Ann(\C(\Sub(\Pi/I_x)))=I_x$, and 
$\Ann(\C(\Pi/I_v))=I_v$, so $I_x=I_v$, and thus $x=v$ by
\cite[Proposition III.1.9, Proposition III.3.5]{BIRS}.  
%\comment{Adapt to Dynkin case.
%HT says: wait, the point here is that we have two different expressions
%for the same subcategory, so they will actually have the same 
%annihilator.  Therefore it seems to me we don't need to invoke BIRS or
%an adaptation of it.}
It follows that $\C(\Pi/I_x)$ is
subclosed for any $x$, and by Proposition \ref{prop.dualities}, we
conclude that $\C(I_w)$ is quotient closed for all $w\in W$.  
\end{proof}

\begin{proof}[Proof of Theorems \ref{two.two} and Theorem~\ref{two.three}] 
Theorem~\ref{main} shows that the correspondences of Theorem~\ref{two.two}
yield a bijection between the cofinite quotient closed subcategories 
of $\mathcal P$ and some subset $X$ of $W$.  Theorem~\ref{converse} 
shows that for any $w\in W$, $\C(I_w)$ is quotient closed.  It therefore
can also be written as $\C(I_x)$ for some $x \in X$.  From the fact that
$\C(I_w)=\C(I_x)$ we conclude that $I_w = I_x$ (using Theorem~\ref{four.seven}
in the non-Dynkin case, and Lemma~\ref{notesDDynkin} in the Dynkin case).  
It then follows from \cite[Proposition III.1.9, Proposition III.3.5]{BIRS} that $w=x$.  
Therefore, $w\in X$. Thus
$X=W$, and Theorem~\ref{two.two} is proved.  
 
Theorem~\ref{two.three} now also follows.\end{proof}

\section{Infinite words}
In this section we extend our bijection between the Weyl group and
the cofinite quotient closed subcategories of $\mathcal P$, to a 
bijection between a specific class of subwords of ${\underline c}^\infty$ and 
arbitrary (i.e., not necessarily cofinite) quotient closed subcategories of $\mathcal P$.   

Let $Q$ be non-Dynkin.  
Fix the (one-way) infinite word ${\underline c}^\infty=\underline c\,
\underline c\, \underline c\dots$.  

We say
that an infinite subword ${\underline w}$ of ${\underline c}^\infty$ is {\it leftmost} if,
for all $n$, the subword of ${\underline c}^\infty$ consisting of the 
first $n$ letters of ${\underline w}$ is leftmost (among all reduced words 
in ${\underline c}^\infty$ for that element of $W$ --- i.e., in the usual sense).  

\begin{theorem} \label{noncofinite}
There is a bijective correspondence between the
leftmost subwords of ${\underline c}^\infty$ and the quotient closed subcategories
of the preprojective component of $\mod kQ$: the 
reflections in the word correspond to indecomposable objects not in the
subcategory.     
\end{theorem}

\begin{proof} We need only worry about infinite words and non-cofinite
subcategories.  Let $\mathcal C$ be a non-cofinite quotient closed 
subcategory.  It determines a sequence $\mathcal C_1,\mathcal C_2,\dots$
where $\mathcal C_i$ consists of the indecomposable 
$kQ$-modules except for the $i$ leftmost indecomposables missing from
$\mathcal C$.  
Clearly, $\mathcal C_i$ is cofinite and 
quotient closed.  It therefore determines a word $w_i$.  By construction,
$w_{i-1}$ is a prefix of $w_i$, and thus they together define an infinite
word all of whose prefixes are leftmost, which means that it is, by 
definition, an (infinite) leftmost word in ${\underline c}^\infty$.  

The argument in the converse direction works in essentially the same way.  
Each finite prefix
determines a subcategory which is quotient closed; the intersection of
all these is a non-cofinite quotient closed subcategory. 
\end{proof}

Our main theorem, Theorem~\ref{main}
relates quotient closed subcategories, elements of $W$,
and certain ideals in $\Pi$.  One might therefore wonder if it is possible
to find a class of ideals of $\Pi$ in bijection with the subcategories
and words of the previous theorem.  The obvious way to do this fails.  
Specifically, let $w_i$ be the element of $W$ 
corresponding to the first $i$ symbols of an infinite subword 
${\underline w}$ of ${\underline c}^\infty$, and then take 
$I_{\underline w}=\bigcap I_{w_i}$.  It can happen that if ${\underline w}$ and ${\underline v}$ 
are distinct leftmost subwords of ${\underline c}^\infty$, then nonetheless 
$I_{\underline w}=I_{\underline v}$.  
For example, any leftmost 
subword ${\underline w}$ with
the property that each simple reflection occurs an infinite number
of times, will yield $I_{\underline w}=0$.  (Consider the case of
the quiver $\widetilde A_1$, with two arrows from vertex 1 to vertex 2.  
Now consider the infinite words $s_1s_2s_1s_2s_1\dots$ and $s_2s_1s_2s_1s_2\dots$.
The subcategory corresponding
to the first of these words is empty, while that corresponding to the
second word contains the simple projective. Both words nonetheless define the
$0$ ideal.  Further, note that there is no ideal $I$ of $\Pi$ such that 
$\mathscr C(I)$ is the additive hull of the simple projective.)

\medskip
Theorem \ref{noncofinite} gives a correspondence between certain subcategories and
certain subwords of 
$\underline{c}^\infty$.  This bijection seems somewhat different from the
bijection in Theorem~\ref{main}, in that, in that theorem, 
we biject
subcategories with elements of $W$, rather than with subwords of 
$\underline{c}^\infty$.  
However, theorem~\ref{main} could be formulated in a fashion parallel to
that of Theorem~\ref{noncofinite}, 
because every element of $W$ has a unique leftmost
expression as a subword of $\underline{c}^\infty$.  

We will now proceed instead 
to reformulate Theorem \ref{noncofinite} in a way which
is closer to Theorem \ref{main}.  In order to do so, we will replace the
leftmost subwords of $\underline c^\infty$ which appear in 
Theorem \ref{noncofinite} by certain equivalence classes of 
subwords of $\underline c^\infty$, 
such that the equivalence classes of finite subwords 
are just the reduced subwords in $\underline c^\infty$ for a given $w \in W$.    

We say that an infinite word is {\it reduced} if any prefix of it is
reduced.  We restrict our attention to such words.  

Say that the infinite reduced word ${\underline v}$ is a {\it braid limit} of the 
infinite reduced word ${\underline w}$ if 
there is some, possibly infinite, sequence
of braid moves $B_1,\dots$, which transforms ${\underline w}$ into ${\underline v}$, 
such that, 
for any particular position $n$, there is some $N(n)$ such that $B_j$ for
$j>N(n)$ only affects positions greater than $n$.  (This is a rephrasing
of the definition in \cite{LP}.)

Note that it is possible for ${\underline v}$ to be a braid limit of ${\underline w}$
even if the converse is not true.  An example in $\widetilde A_2$ 
(from \cite{LP}) is as
follows.  Let ${\underline w}=s_2s_1s_2s_3s_1s_2s_3\dots$ and let 
${\underline v}=s_1s_2s_3s_1s_2s_3\dots$.  
Transform ${\underline w}\to s_1s_2s_1s_3s_1s_2s_3\dots \to
s_1s_2s_3s_1s_3s_2s_3\dots$.  After $i$ steps, the first $2i$ positions
agree with ${\underline v}$.  However, there are no braid moves applicable to
${\underline v}$, thus ${\underline w}$ is certainly not a braid limit of ${\underline v}$.  
%It is also straightforward to establish that these two elements have
%different inversion sets (taking the natural extension of the notion
%for elements of $W$). 

We then have the following proposition, analogous to the statement that
any element of $W$ has a unique leftmost expression in $\underline{c}^\infty$.

\begin{proposition} Any infinite reduced word has a unique leftmost braid limit.
\end{proposition}

\begin{proof}
Let $s_1$ be the first reflection in ${\underline c}^\infty$.  
\cite[Lemma 4.8]{LP} (an extension to infinite words of the usual Exchange
Lemma from Coxeter theory) states that, given an infinite word ${\underline w}$, one of
two things will happen when we consider the infinite word ${s_1\underline w}$:
either it will be reduced in turn, or there will be a unique reflection from
${\underline w}$ which cancels with $s_1$, leaving some 
${\underline {\hat w}}$.  

In the first case, no finite prefix of 
$\underline w$ is equivalent under braid moves 
to a word beginning $s_1$.  Since a finite
number of braid moves can only alter a finite prefix of ${\underline w}$, 
it follows that no braid limit for $\underline w$ can begin with $s_1$.  
In the second case, ${\underline w}$ is equivalent,
after a finite number of braid moves, to $s_1{\underline {\hat w}}$.  
Clearly, in this case, ${\underline w}$ admits a braid limit which begins
with $s_1$.

Therefore, if the first case holds, no braid limit for ${\underline w}$ can
involve the initial $s_1$, so it plays no role and we can continue on
to consider the next simple
reflection in ${\underline c}^\infty$.  In the second case, a finite number
of braid moves suffice to bring $s_1$ to the front of ${\underline w}$, and we 
can now go on to find a braid limit for ${\underline {\hat w}}$ beginning with
the second reflection in ${\underline c}^\infty$.   
\end{proof}

Say that two (possibly infinite) reduced words in the simple generators of $W$ 
are equivalent if they have the same leftmost braid limit in
$\underline{c}^\infty$.  (The equivalence classes in which the words are 
of finite length correspond naturally to elements of $W$.)
Then we can restate Theorem~\ref{noncofinite}
in the following way:

\begin{corollary} There is a bijection between equivalence classes of
reduced words in the simple generators of $W$ and 
quotient closed subcategories of the preprojective component of mod $kQ$.  
\end{corollary}

\section{Quotient closed subcategories of modules over hereditary algebras over finite
fields}
In this section, we show how to extend our analysis of quotient closed
subcategories of modules for path algebras over algebraically closed fields 
to quotient closed subcategories of modules for arbitrary hereditary
algebras over finite fields.  

Let $\F$ be a finite field with $q$ elements, and 
let $\oF$ be its algebraic closure.  
Let $H$ be a hereditary algebra over $\F$.  We first recall (see 
Subsection \ref{recsec} below) that
there is a quiver $Q$ and an $\F$-endomorphism $F$ of $\oF Q$ called a 
Frobenius morphism, such that $H\cong (\oF Q)^F$, the $\F$-subspace of 
$\oF Q$ fixed under $F$.  
We then invoke a theorem which says
that the module category of $(\oF Q)^F$ is equivalent to the category of
$F$-stable representations of $\oF Q$.  (The $F$-stable representations
of $\oF Q$, defined below, are a certain non-full subcategory of the
representations of $\oF Q$.)
We apply our analysis of
quotient closed categories for path algebras over algebraically closed
fields to analyze quotient closed subcategories of the
$F$-stable representations of $\oF Q$.  

The Frobenius morphism $F$ comes
from a certain quiver automorphism $\sigma$ of $Q$.  There is a Weyl group
$W_{Q,\sigma}$ 
(typically not simply laced) associated to the quotient of $Q$ by
$\sigma$.  We prove Theorem \ref{fqmain}, an analogue of our Main Theorem,
which applies to
cofinite quotient closed subcategories of $H$-modules. 

Our main reference for the techniques specific to this section
is \cite[Chapters 2 and 3]{DDPW}.

\subsection{Recollection of results on Frobenius twisting}\label{recsec}
%Let $\f$ denote the field automorphism of $\oF$ taking 
%$\lambda$ to $\lambda^q$.  So $\oF^\f=\F$.  
%
Let $V$ be an $\oF$ vector space. 
An $\F$-linear isomorphism $F:V\to V$ is called a Frobenius map on $V$ if:
\begin{itemize}
\item $F(\lambda v)=\lambda^q v$ for $\lambda \in \oF$.
\item for any $v \in V$, $F^t(v)=v$ for some $t\geq 1$.  
\end{itemize}

Let $A$ be an $\oF$-algebra.  A
Frobenius morphism on $A$ is a Frobenius map on the underlying vector space of 
$A$ which also satisfies $F(ab)=F(a)F(b)$ for $a,b$ in $A$. We write 
$A^F$ for the elements of $A$ fixed by $F$.  It is an $\F$-subalgebra of $A$.   

For example, let $Q$ be a quiver with an automorphism $\sigma$.  
Define a Frobenius morphism $F=F_\sigma: \oF Q\to \oF Q$ by sending
$\sum x_ip_i$ to $\sum x_i^q\sigma(p_i)$, for $x_i \in \oF$ and $p_i$ paths
in $Q$.  %(This is \cite[Example 2.7]{DDPW})
In fact, this example plays an important role.  

\begin{theorem} [{\cite[Theorem 3.40]{DDPW}}] \label{setupt}
Any hereditary algebra $H$ over 
$\F$ can be realized as $(\oF Q)^{F_\sigma}$ for a suitable quiver $Q$ and
quiver automorphism $\sigma$.  
\end{theorem}

%From now on, we fix this choice of $Q$, $\sigma$, and $F$.  

%We will explain how to use $\sigma$ to 
%define a Frobenius morphism $F$ of $\oF Q$, so that the
%$F$-stable representations over $\oF$ are equivalent to the representations
%of the $\F$-algebra $(\oF Q)^F$.  Any hereditary algebra over a finite
%field can be realized as $(\oF Q)^F$ for some $Q$ and $\sigma$.  

%We will then show that quotient closed categories of $(\oF Q)^F$-modules
%can be described in terms of our usual Coxeter combinatorics, in the 
%non-simply-laced Weyl group corresponding to $Q/\sigma$.  

%{\it Review of Frobenius twisting.}
%I am following closely \cite[Chapter 2]{DDPW}.  I am using left
%modules because they do, and all my modules and algebras
%are finite dimensional. 

Let $(A,F_A)$ be an algebra with a Frobenius morphism.  Let $M$ be an 
$A$-module, with a Frobenius map $F_M$ (i.e., a Frobenius map on the 
underlying vector space of $M$).  

Define a new $A$-module, $M^{[F_M]}$, which is called the 
\emph{Frobenius twist} of $M$,  
by letting the underlying vector space be that of $M$,
and defining the $A$-action by:

$$m * a = F_M(F_M^{-1}(m)F_A^{-1}(a)) $$

Up to isomorphism, this $A$-module is independent of the choice of $F_M$.

A module $M$ is called $F$-stable if for some Frobenius
map $F_M$ (or, equivalently, for all of them) we have
$M\cong M^{[F_M]}$.  If $M$ is $F$-stable, it is possible to choose $F_M$ so
that $M=M^{[F_M]}$, or equivalently 
\begin{equation}\label{goodF} F_M(ma)=F_M(m)F_A(a).\end{equation}  
The $F$-stable modules form a category: its objects
are pairs $(M,F_M)$ with $F_M$ satisfying (\ref{goodF}), and the morphisms
from $(M,F_M)$ to $(N,F_N)$ are maps $f\in \Hom(M,N)$ satisfying
$f \circ F_M=F_N\circ f$.  

Now we can state the following theorem:
\begin{theorem}[{\cite[Theorem 2.16]{DDPW}}] \label{equiv}
The category of $F$-stable $A$-modules is equivalent to the category of
$A^F$-modules. 
\end{theorem} 
%(If $M$ is $F$-stable, then there exists a Frobenius map
%$F_M$ such that $F_M(am)=F_A(a)F_M(m)$.  Then
%$M^F=\{m\in M\mid F_M(m)=m\}$
%is a module over $A^F$.)

We want to understand the indecomposable $F$-stable $A$-modules.  It turns
out that they can be constructed as follows.  
Let $M$ be an
$A$-module with Frobenius morphism $F_M$. 
Let $r$ be the maximum possible so that 
$M$, $M^{[F_M]}$, $M^{[F_M]^2},\dots$, $M^{[F_M]^{r-1}}$ are pairwise non-isomorphic. 
(Such an $r$ exists by \cite[Proposition 2.13]{DDPW}.) We
can choose $F_M$ so that $M=M^{[F_M]^{r}}$ (note that here we want equality, 
not just isomorphism).
Let 
$$\widetilde M= M\oplus  M^{[F_M]}\oplus  M^{[F_M]^2} \oplus \cdots 
\oplus M^{[F_M]^{r-1}},$$

and define $F_{\widetilde M}$ by 
$$F_{\widetilde M}(m_0,\dots,m_{r-1})=
(F_M(m_{r-1}),F_M(m_0),\dots F_M(m_{r-2})).$$

\begin{theorem}[{\cite[Theorem 2.20]{DDPW}}] \label{indec}
Let $M$ be an indecomposable $A$-module.  Then 
$\widetilde M$ is an indecomposable $F$-stable module, 
and every
indecomposable $F$-stable module arises in this way.  
\end{theorem}

The AR-sequences of $A$-mod and $A^F$-mod are closely related:
\begin{theorem}[{\cite[Theorem 2.48]{DDPW}}]
Every almost split sequence
$$0\to X \to Y \to Z \to 0$$ in 
$A$-mod gives rise to an almost split sequence
$$0\to \wt X^F \to \wt Y^F \to \wt Z^F \to 0$$
in $A^F$-mod and every almost split sequence in 
$A^F$-mod arises in this way.
\end{theorem}

In particular, in the case we considered above, where $A=\oF Q$ and $F=F_\sigma$
for $\sigma$ an automorphism of $Q$, 
there is a nice description of the preprojective 
component of $A^F$-mod.

\begin{theorem}[{\cite[Theorem 3.30]{DDPW}}]
The preprojective component of the AR quiver of 
$(\oF Q)^F$ is (up to multiplicities of arrows) a translation quiver on 
$Q/\sigma$, whose vertices are labelled by $\sigma$-orbits of 
preprojective indecomposable representations of $\oF Q$.  
\end{theorem}

The cited theorem provides a more precise statement about multiplicities of arrows, but we will not need it.  

\subsection{Quotient closed subcategories of hereditary algebras over
$\F$}

Let $H$ be a hereditary algebra over $\F$.  Fix a quiver $Q$, 
a quiver automorphism $\sigma$, and a Frobenius map
$F=F_\sigma$ as in Theorem \ref{setupt}, so $H\cong (\oF Q)^F$.  
By Theorem \ref{equiv},
the category of modules over $H$
is equivalent to the category of $F$-stable modules over $\oF Q$.

\begin{proposition}
There is a numbering of the vertices of $Q$ which has the following
properties:
\begin{enumerate} 
\item If $\Hom(P_i,P_j)\ne 0$ then $i<j$, and
\item The labels assigned to any 
$\sigma$-orbit of vertices form a consecutive sequence of 
numbers.  
\end{enumerate}
\end{proposition}

\begin{proof} A labelling $f$ satisfying (1) exists because
$Q$ has no oriented cycles.  Define a new labelling $\wt f$, by averaging
$f$ over $\sigma$-orbits.  
Now $\wt f$ is constant on $\sigma$-orbits and still
satisfies (1), since if there is an arrow from $[i]$ to $[j]$ in 
$Q/\sigma$, then $\wt f([j])-\wt f([i])$ is the average 
difference between
the head and the tail of the arrows in the $\sigma$ orbit of this arrow.  
Number the
vertices of $Q$ in an order such that $\wt f$ weakly increases at each step, 
and such that once we assign a
label to an element of a $\sigma$-orbit, we go on to assign consecutive
labels to the remainder of the $\sigma$-orbit before we assign any further
labels.  This labelling satisfies (1) and (2).  
\end{proof}

Fix a numbering for $Q$ as in the proposition.  
Consider the subgroup of the Weyl group $W_Q$ generated by the elements
$s_{[i]}$, where $[i]$ denotes the $\sigma$-orbit of $i$, and 
$s_{[i]}$ is the product of the simple reflections labelled by elements of
$[i]$, in arbitrary order.  Since the corresponding vertices
 are not adjacent in $Q$,
$s_{[i]}$ does not depend on the choice of the order in which the product
is taken.  

The subgroup of $W_Q$ generated by the $s_{[i]}$ is again a Weyl group \cite{St};
we denote it by $W_{Q,\sigma}$.  It is typically not simply laced.  Let 
$\wt c$ denote the Coxeter element in $W_{Q,\sigma}$ corresponding to this 
ordering.

We can now state the main theorem of this section:

\begin{theorem}\label{fqmain} Let $H$ be a hereditary algebra over $\F$.  Let $Q$ and $\sigma$ be defined as above.   
Then quotient closed cofinite subcategories of $H$-modules 
correspond to 
lexicographically first subwords of $\wt {\underline c}^\infty$ in $W_{Q,\sigma}$.  
\end{theorem}

Let $\mathcal C$ be a cofinite quotient closed subcategory of $H$-modules.
We think of $\mathcal C$ as a category of $F$-stable modules over $\oF Q$.

Write $\wc$ for the category of
modules over $\oF Q$ obtained from $\mathcal C$ by forgetting
the $F$ structure, and then taking direct sums of direct summands.  
We now want to describe $\wc$.  

\begin{proposition}\begin{enumerate}
\item The indecomposable objects of $\wc$ are a union of $\sigma$-orbits
of indecomposable representations of $Q$.  

\item $\wc$ is a cofinite, quotient closed subcategory of $\oF Q$-mod.  
\end{enumerate}
\end{proposition}

\begin{proof}
The first claim follows from Theorem \ref{indec}.

Note that the second claim
is not obvious, because
$\wc$ has more morphisms than $\mathcal C$.  

So, suppose that we have a surjection $f$ from $X$ to $Y$, with $X \in 
\wc$. We may assume that $X$ and $Y$ are (sums of) preprojective
representations of $Q$.  
Let $K$ be the kernel of $f$.

Since $X$ is a sum of preprojective representations, it can be lifted
from a $\F$-representation, and thus admits a Frobenius map $F_X$ with
finite period $r$.  

We then have an $F$-stable module 
$$\widetilde X = X \oplus X^{[F_X]} \oplus \dots \oplus X^{[F_X]^{r-1}}$$

Consider the submodule defined by 
$\widetilde K = K \oplus K^{[F_X]} \oplus \dots \oplus K^{[F_X]^{r-1}}$.  
This is also $F$-stable.
We now see that $Y$ is a summand of $\widetilde X/\widetilde K$.  This
implies that $\wc$ is quotient closed.  
\end{proof}

We also have the following converse:

\begin{proposition}\label{conprop}
Any cofinite quotient closed subcategory of modules over 
$\oF Q$, 
whose indecomposables consist of 
a union of $\sigma$-orbits, arises as 
$\wc$ for some cofinite quotient closed category $\mathcal C$
of $F$-stable modules.
\end{proposition}

\begin{proof} By Theorem \ref{indec}, there is a subcategory 
$\mathcal C$ of $F$-stable modules such that the associated subcategory
of $\F Q$-modules is $\wc$.  Since $\wc$ is quotient closed, so is
$\mathcal C$. \end{proof}

We now prove the main theorem of this section.  

\begin{proof}[Proof of Theorem \ref{fqmain}] A cofinite 
quotient closed subcategory
of $H$-mod corresponds, as above, to a cofinite, quotient closed 
subcategory of $\oF Q$-mod.  Thus, when we read the word in 
$W_Q$ corresponding to it, it is a leftmost word in ${\underline c}^\infty$.  (Here
$c$ is the Coxeter element in $W_Q$, and $\underline c$ is the 
corresponding word.)  Because the indecomposables in
the subcategory of $\oF Q$-mod are unions of $\sigma$-orbits, the factors
in the word can be grouped into generators of $W_{Q,\sigma}$, so we can view
it as a word in $\wt {\underline c}^\infty$, which is automatically also leftmost.  

Conversely, suppose we take some $w\in W_{Q,\sigma}$. 
Because the leftmost word for $w$ inside ${\underline c}^\infty$ can be
obtained greedily, for each $\sigma$-orbit of reflections in ${\underline c}^\infty$, 
either all of them will be used in the leftmost word for $w$, or none of 
them will.  Thus the leftmost word for $w$ in ${\underline c}^\infty$ corresponds to a 
word for $w$ in $\wt {\underline c}^\infty$, and therefore to the leftmost word there.
The subcategory of $\oF Q$-mod corresponding to $w$ satisfies the
hypotheses of Proposition \ref{conprop}, so it corresponds to a 
cofinite, quotient closed subcategory of $F$-stable $\F Q$-modules, as
desired.  \end{proof}

\section{Subclosed subcategories}

We have seen that $\C$ induces a bijection between the ideals of the form $I_w$ in $\Pi$ and the cofinite quotient closed subcategories of $\mod kQ$. It is natural to ask if $\C$ similarly induces a bijection between the quotients $\Pi / I_w$ and certain subcategories of $\mod kQ$, and further if one can explicitly describe the subcategories of $\mod kQ$ which are of the form $\C(\Pi/I_w)$ for some $w$.

We start by observing that in case $Q$ is Dynkin the situation is as good as one could have hoped:

\begin{theorem} \label{thm.Dynkin_subclosed}
Let $Q$ be a Dynkin quiver. Then the map
\begin{align*}
 W & \to \{\text{subcategories of } \mod kQ\} \\
 w & \mapsto \C(\Pi/I_w)
\end{align*}
induces a bijection between $W$ and the subclosed subcategories of $\mod kQ$.
\end{theorem}

\begin{proof}
We have
\[ \C(\Pi / I_w) = \C(D I_{w_0 w^{-1}}) = D \C_{\rm left}(I_{w_0 w^{-1}}) \]
from Section~\ref{sectionsix}.  
Now the claim follows since $w \mapsto w_0 w^{-1}$ is a bijection from $W$ to itself, while $w \mapsto \C_{\rm left}(I_w)$ is a bijection from $W$ to the set of 
quotient closed subcategories of $\mod {kQ}^{\op}$, and $D$ induces a bijection between quotient closed subcategories of $\mod {kQ}^{\op}$ and subclosed subcategories of $\mod kQ$.
\end{proof}

The most obvious guess would be that for $Q$ arbitrary, the map $w \to \C(\Pi/I_w)$ might be a bijection from $W$ to the subclosed subcategories of $\mod kQ$ containing only finitely many indecomposables. However, this is not the case: there are subclosed subcategories of $\mod kQ$ with finitely many indecomposable objects which do not appear as $\C(\Pi/I_w)$ for any $w$. For instance, let 
$Q$ be the Kronecker quiver. The subcategories consisting of direct
sums of copies of one quasi-simple regular module and the simple projective module are subclosed, but not of the form $\C(\Pi/I_w)$, as we will see in 
Proposition~\ref{twovertices}.

We however suspect that the following statement in the converse 
direction holds:

\begin{conjecture} \label{conj.is_subclosed}
For $w \in W$, the subcategory $\C(\Pi/I_w)$ of $\mod kQ$ is subclosed.
\end{conjecture}

By Theorem~\ref{thm.Dynkin_subclosed} the conjecture holds for $Q$ Dynkin. It is also easy to verify this conjecture in the case that $Q$ has two vertices
by a direct calculation.  See Proposition \ref{twovertices} below for
an explicit description of the categories that arise.

Recall that we have seen in Lemma~\ref{subclosed} that the categories $\C(\Sub \Pi/I_w)$ are always subclosed. It follows that $\Sub \C(\Pi/I_w) = \C(\Sub \Pi/I_w)$, and hence that Conjecture~\ref{conj.is_subclosed} is equivalent to
\[ \C(\Sub \Pi/I_w) = \C(\Pi/I_w). \]

We now formulate two conjectures on the description of the subcategories 
$\C(\Pi/I_w)$. In the first one we restrict to the affine case, where the combinatorial description is somewhat simpler.

\begin{conjecture} \label{affinedescription}
Suppose that $Q$ is affine.  
A full subcategory $\mathcal Z$ of $kQ$-mod arises as $\mathscr C(\Pi/I_w)$
for some $w\in W$ iff
\begin{itemize}
\item $\mathcal Z$ has a finite number of indecomposables, 
\item $\mathcal Z$ is subclosed, and
\item for any tube, there is at least one ray in the tube which
does not intersect $\mathcal Z$.
\end{itemize}
\end{conjecture}

In order to generalize this conjecture to arbitrary quivers, we introduce the following notion.
Define the reduced $\Ext$-quiver of a full subcategory $\mathcal Z$ of
$kQ$-mod as follows: The vertices are the indecomposable objects of $\mathcal Z$.  There is an arrow from $Y$ to $X$ if the simple $\mathcal{Z}$-module $\frac{\Hom_{kQ}(-, X)}{{\rm Rad}(-,X)}$ is a direct summand of the socle of the $\mathcal{Z}$-module $\Ext_{kQ}^1(-, Y)$. Equivalently there is an arrow from $Y$ to $X$ if there is a morphism from $\tau^{-1} Y$ to $X$ which does not factor through any radical map $\mathcal{Z} \to X$.

\begin{conjecture} \label{description}
A full subcategory $\mathcal Z$ of $kQ$-mod arises as $\mathscr C(\Pi/I_w)$
for some $w\in W$ iff
\begin{itemize}
\item $\mathcal Z$ has a finite number of indecomposables, 
\item $\mathcal Z$ is submodule-closed, and
\item the reduced $\Ext$-quiver of $\mathcal Z$ contains no cycles. 
\end{itemize}
\end{conjecture}

This conjecture is clearly true for $Q$ of finite type.  (The third condition
is vacuous in this case.)

In the tame case, we show that the two conjectures above coincide.

\begin{proposition} In the case that $Q$ is affine, Conjecture 
\ref{affinedescription} is equivalent to Conjecture \ref{description}.
\end{proposition}

\begin{proof}
Suppose that $\mathcal Z$ is a subclosed category, 
and that there is some tube such that 
$\mathcal Z$ contains objects
from each ray of the tube.  By the fact that $\mathcal Z$ is subobject
closed, it contains the quasisimples from the bottom of the tube.  Since
each is the AR-translation of the next around the tube, the corresponding
extensions cannot factor through any other element of $\mathcal Z$, so
they give rise to a cycle in the reduced $\Ext$-quiver of $\mathcal Z$,
contrary to our assumption.  

Conversely, suppose that $\mathcal Z$ is subclosed and 
each tube has a ray such that $\mathcal Z$ does
not intersect that ray.  Suppose $Y$ and $X$ are in the same tube, with
$X$ strictly higher than $Y$.  A map from $\tau^{-1} Y$ to $X$ factors 
through $X'$, the indecomposable in the same ray as $X$ and on the same
level as $Y$, and $X'$ is in $\mathcal Z$ since $\mathcal Z$ is subclosed.
Therefore, there is no arrow in the reduced $\Ext$-quiver from $Y$ to $X$, 
so the only arrows in the reduced $\Ext$-quiver go from an object to another
object at the same height or lower.  
Further, there can only be an arrow from 
$Y$ to some 
$X$ at the same height as $Y$ if $X=\tau^{-1} Y$, since otherwise the map 
from $\tau^{-1} Y$ will factor through $X''$, the indecomposable on the
same ray as $X$ which is one level lower.  

Thus, a cycle would necessarily involve objects all at the same height,
and each would have to be $\tau^{-1}$ of the previous one.  This would imply
that there was no ray in the tube not intersecting $\mathcal Z$.  
\end{proof}

We now prove Conjecture~\ref{description} 
for the case that $Q$ is a quiver with 
two vertices.  

\begin{proposition} \label{twovertices}
Conjecture \ref{description} holds when $Q$ is a
quiver with two vertices. \end{proposition}

\begin{proof} The subcategories that arise as $\mathscr C(\Pi/I_w)$ are 
exactly those of the following form:
\begin{itemize}
\item A finite initial segment of the preprojective component, or
\item A finite initial segment of the preprojective component 
together with the simple injective.
\end{itemize}

It is clear that these subcategories satisfy the conditions of 
Conjecture \ref{description}, so it is just a matter of checking that
no other subcategory does.  

Let $\mathcal Z$ be some subcategory satisfying the conditions of 
Conjecture \ref{description}.  If $\mathcal Z$ contains a non-simple 
injective, then (being subclosed) it would also contain all the
predecessors of $\mathcal Z$ in the preinjective component, and thus
it would not be finite, contradicting our assumption.

Suppose now that $\mathcal Z$ contains some regular objects, $R_1,\dots,
R_r$.  Since $kQ$-mod has no rigid regular objects, each of these objects
admits a self-extension.  Thus, for each $R_i$, there is a map from
$\tau^{-1}R_i$ to $R_i$.  This does not necessarily yield an arrow in the 
reduced $\Ext$-quiver, but we can conclude that there is {\it some} arrow 
in the reduced $\Ext$-quiver of $\mathcal Z$ starting at $R_i$, and further
that this arrow goes to a regular indecomposable of $\mathcal Z$,
since the morphism from $\tau^{-1}R_i$ to $R_i$ factors through the 
morphism corresponding to this arrow.  Thus,
the reduced $\Ext$-quiver contains arbitrarily long walks, so it must
contain a directed cycle, contradicting our assumption.

Finally, if $\mathcal Z$ contains a preprojective object $E$, it must
contain all the predecessors of $E$, since $\mathcal Z$ is subclosed.  
It follows that the only subcategories $\mathcal Z$ which satisfy the
conditions of Conjecture \ref{description} are those which we have already
identified.  \end{proof}

\section{Which cofinite quotient closed subcategories are torsion classes?}

We have established that the cofinite quotient closed subcategories of 
$kQ$-mod can be formed as the additive hull of $\mathscr C(I_w)$ together
with all non-preprojective indecomposable objects for $w\in W$.  
It is natural to ask
for which $w$ these subcategories are actually torsion classes.  

When we have found a torsion class, it is also natural to ask about the 
corresponding torsion-free class.  Since our torsion classes are cofinite,
the corresponding torsion-free class will be finite, and it will therefore
be useful to recall the correspondence established in \cite{IT,AIRT}
between torsion-free classes and 
certain elements of $W$ called $c$-sortable elements.

As usual, $c=s_1s_2\dots s_n$, where the simple reflections 
(or equivalently, the vertices of $Q$) are numbered
compatibly with the orientation of $Q$.  An element $w\in W$ is called 
$c$-sortable if 
there is an expression for $w$ of the form $c^{(0)}c^{(1)}\dots c^{(r)}$, 
where each $c^{(i)}$ consists of some subset of the simple reflections, taken
in increasing order, and such that the set of reflections appearing in
$c^{(i)}$ is contained in the set of reflections appearing in $c^{(i-1)}$ \cite{Re}.  
It is shown in \cite{AIRT} that there is a one-to-one correspondence
between $c$-sortable elements of $W$ and finite torsion-free classes for 
$kQ$-mod, which takes $w$ to $\mathscr C(\Pi/I_w)$.

%\comment{check flow around this paragraph}
%Let $c$ be a Coxeter element of $W,$ that is, a product of $\seq{s}$ in some order. Assume that $c$ is compatible with the orientation of $Q.$ Then

%An element $w\in W$ is \emph{$c$-sortable} if there is some reduced expression $\mathbf{w} = c^{(0)}c^{(1)}\ldots c^{(r)},$ where $c\supseteq c^{(0)}\supseteq c^{(1)}\supseteq\ldots\supseteq c^{(r)}.$ Then there is a one-one correspondence between $c$-sortable words in $W$ and finite torsion-free classes in $\m{kQ}$ \cite{T,AIRT}.

%We know that $\C(\Sub \Pi/I_w)$ is a torsion-free class 
%if and only if $w$ is $c$-sortable.  It is then natural to
%try to describe the elements $w$ such that $\C(I_w)$ is
%a torsion class, and in that case, to describe the $c$-sortable
%word corresponding to the associated torsion-free class.  

Therefore, when $\mathscr C(I_w)$ together with the non-preprojective
indecomposables form a torsion class, we can ask for the element $v\in W$
such that the corresponding torsion-free class is given by $\mathscr C(\Pi/
I_v)$.  
In this section we give conjectural answers to both these questions,
and we prove that our conjectures hold in the Dynkin case.  

Write $\sort_c(w)$ for the longest $c$-sortable prefix of $w$.  

\begin{conjecture} \label{whentorsion}
The following conditions are equivalent:
\begin{enumerate}
\item the additive category generated by $\mathscr C(I_w)$ together with all
non-preprojective indecomposable $kQ$-modules
is a torsion class,
\item for every $i$ such
that $\ell(ws_i)>\ell(w)$, we have that $\sort_c(ws_i)$ is 
strictly longer than $\sort_c(w)$.
\end{enumerate}
\end{conjecture}

Consider $A_2$, with $S_1$ the simple projective, so $c=s_1s_2$.  
We find that the elements of $W$ satisfying each of the above conditions 
are
$e$, $s_1$, $s_1s_2$, $s_2s_1$, $s_1s_2s_1$.  (For example, $s_2$ does
not satisfy the second condition, because $s_2s_1$ is longer than $s_2$,
but the longest $c$-sortable prefix of both $s_2s_1$ and of $s_2$ is
$s_2$.  On the other hand, note that $s_2s_1$ satisfies the conditions:
its longest $c$-sortable prefix is $s_2$, while the only word 
which can be obtained 
by lengthening $s_2s_1$ is $s_2s_1s_2=s_1s_2s_1$ which is $c$-sortable.)
 
\begin{conjecture}\label{associated}
If the additive hull of $\mathscr C(I_w)$ together with the non-preprojective
indecomposable objects forms a torsion class, its corresponding 
torsion-free class is that associated to $\sort_c(w)$.  
\end{conjecture}

We will now prove both these conjectures for finite type.  In order to
do so, we introduce some notation.  

There is an order on $W$, called {\it right weak order}, in which 
$u\leq_{\rm R} v$ iff there is a reduced expression for $v$ with a prefix which is
an expression for $u$.  This is a weaker order than Bruhat order,
in the sense that if $u \leq_{\rm R} v$, it is also true that
$u\leq v$ in Bruhat order.  (Left weak order, which we shall not need here,
is defined similarly, using suffixes instead of prefixes.)  For more on
weak orders, see \cite[Chapter 3]{BB}.

In finite type, the 
map $\sort_c$, which takes $W$ to $c$-sortable elements, is a 
lattice homomorphism from $W$ with the right weak order to the 
$c$-sortable elements of $W$, ordered by the restriction of right weak 
order  \cite[Theorem 1.1]{Re:au}.
This implies, in particular, that each fibre of 
$\sort_c$ is an interval in $W$.  

\begin{lemma} \label{equivalences}
For $Q$ of finite type, the following conditions on
$w\in W$ are equivalent:
\begin{enumerate}
\item For every simple reflection $s_i$ such that $\ell(ws_i)>\ell(w)$, 
we have that $\ell(\sort_c(ws_i))>\ell(\sort_c(w))$. 
\item $w$ is the unique longest element among those $x\in W$ satisfying
$\sort_c(w)=\sort_c(x)$.  
\item $ww_0$ is $c^{-1}$-sortable.
\end{enumerate}
\end{lemma}

\begin{proof}
Suppose (2) does not hold, so there exists some $y>_{\rm R}w$ such that 
$\sort_c(y)=\sort_c(w)$. It then follows that the whole interval from
$y$ to $w$ has the same maximal $c$-sortable prefix, and in particular 
this holds 
for some element $ws_i$ which covers $w$.  This shows that (1) does not hold.

Now suppose that (2) holds.  Let $s_i$ be a simple reflection such that
$\ell(ws_i)>\ell(w)$.  By (2), $\sort_c(ws_i)\ne \sort_c(w)$.  Since 
$ws_i$ lies above $w$ in the right weak order, $\sort_c(ws_i)$ lies above
$\sort_c(w)$ in the right weak order, so in particular it is longer.  This
establishes (1).  

The equivalence of (2) and (3) follows from \cite[Proposition 1.3]{Re:au}.
\end{proof}

\begin{proposition} \label{conjone} Conjecture \ref{whentorsion} holds if 
$Q$ is of finite type.
\end{proposition}

\begin{proof}
We denote by $\mathscr{C}_{\rm left}$ the left module version of $\mathscr{C}$, that is, the map associating to a left $\Pi$-module the category of all finite direct sums of direct summands of its restriction to $kQ$. Note that for a finite dimensional $\Pi$-module $X$ we have $D\mathscr{C}(X) = \mathscr{C}_{\rm left} D(X)$.

By the left module version of \cite{T,AIRT} we know that $\mathscr{C}_{\rm left}(\Pi / I_w)$ is a torsion free class if and only if $w^{-1}$ is $c^{-1}$-sortable. 

Since $D I_w \cong \Pi / I_{w_0 w^{-1}}$ as left $\Pi$-modules by Proposition~\ref{prop.dualities}(1), we have
\[ D\mathscr{C}(I_w) = \mathscr{C}_{\rm left}( \Pi / I_{w_0 w^{-1}} ), \]
and this is a torsion free class in $\mod kQ^{\op}$ if and only if $(w_0 w^{-1})^{-1} = w w_0$ is $c^{-1}$-sortable.

Dualizing we obtain the claim.
\end{proof}

%\begin{proof} 
%By Proposition \ref{prop.dualities}, we know that $I_w$ and  $D(\Pi/I_{w_0w^{-1}})$ are isomorphic
%as $\Pi^\op$-modules.  As in the proof of Lemma~\ref{notesDDynkin},
%we can also think of the $\Pi^\op$-module 
%$D(\Pi/I_{w_0w^{-1}})$ as $(\Pi^\op/I^\op_{ww_0})_{\Pi^\op}$, where  
%we write $I^\op_v$ for the 
%ideal in $\Pi^\op$ defined using the fact that $\Pi^\op$ is itself a 
%preprojective algebra.

%We now restrict to $kQ$-modules, which is the same as 
%applying $\otimes_\Pi kQ$.  
%We get: 

%$$(I_w)_{kQ}= [D(\Pi^\op/I^\op_{ww_0})]\otimes_\Pi kQ
%= D(\Pi^\op/I^\op_{ww_0} \otimes_{\Pi^\op} kQ^\op)$$

%Now, $\C(I_w)$ is torsion iff 
%$\add \Pi^\op/I^\op_{ww_0}\otimes _{\Pi^\op} kQ^\op$ is torsion-free, iff
%$ww_0$ is $c^{-1}$-sortable \cite{AIRT}.  The desired result now follows
%from Lemma~\ref{equivalences}.
%\end{proof}

As was already mentioned, if $w$ is $c$-sortable, 
we know that $\mathscr C_Q(\Pi/I_w)$ is 
torsion-free.  Write $\mathcal F_w$ for this class.

For $c$-sortable $w$, write $\hat w$ for the unique longest word with the
same $c$-sortable prefix as $w$.  
(In order to know that such an element exists, we must continue
to assume that $Q$ is Dynkin.)  Note, in particular, 
$\hat w$ satisfies the equivalent conditions of Lemma~\ref{equivalences}.
Thus, by Proposition \ref{conjone}, 
$\mathscr C(I_{\hat w})$ is a torsion class.
Write $\mathcal T_{\hat w}$ for this subcategory.  From the proof of 
Proposition \ref{conjone}, we also have the further equality:

$$\mathcal T_{\hat w} = \mathscr C(I_{\hat w}) = \mathscr C_{\rm left}
(\Pi/I_{w_0\hat w^{-1}})$$

Proving Conjecture 
\ref{associated} amounts to showing that, 
 for 
$w$ any $c$-sortable element, $(\mathcal T_{\hat w},\mathcal F_w)$ is a 
torsion pair.   

For $w\in W$, choose a reduced expression $w=s_{i_1}\dots s_{i_r}$.  Define
$\Inv(w)$ to consist of the set of positive roots 
$\{ s_{i_1}\dots s_{i_{t-1}}\alpha_{i_{t}}\}$.  Note that this set does not 
depend on the chosen expression for $w$.  %It can also be 
%described as $\Inv(w)=\{\alpha\in \Phi^+\mid w^{-1}(\alpha)\not\in \Phi^+\}$.  
%See \cite[Proposition 4.4.6 and Corollary 1.4.4]{BB}.

\begin{lemma} \label{moreeq} For $v,w$ two $c$-sortable elements, 
the following are equivalent:
\begin{enumerate}
\item $v\leq_{\rm R} w$,
\item $\Inv(v)\subseteq \Inv(w)$,
\item $\mathcal F_v \subseteq \mathcal F_w$,
\item $\hat v \leq_{\rm R} \hat w$,
\item $\hat vw_0 \geq_{\rm R} \hat ww_0$,
\item $\mathcal T_v \supseteq \mathcal T_w$.
\end{enumerate}
\end{lemma}

\begin{proof} The equivalence of (1) and (2) is clear.        The equivalence
of (2) and (3) follows from the
fact that the dimension vectors of the indecomposable objects in 
$\mathcal F_w$ are given by $\Inv(w)$, by \cite[Theorem 2.6]{AIRT}.  
The equivalence of (1) and (4) follows from \cite[Theorem 1.1]{Re:au} 
together with its dual.  
The equivalence of (4) and (5) follows from the fact that 
$\Inv(vw_0)$ is the complement of $\Inv(v)$ in the set of positive roots.
The equivalence of (5) and (6) follows in the same way as the 
equivalence of (1) and (3), using 
$\mathcal T_{\hat v}=\C_{\rm left}(\Pi/I_{w_0v^{-1}})$ and the similar description
of $\mathcal T_{\hat w}$.  
\end{proof}

Define
$\phi$ to be the map on torsion-free classes that takes $\mathcal F_w$
to the torsion-free class associated to $\mathcal T_w$.  We want to
show that $\phi$ is the identity.

\begin{lemma} The map $\phi$ is a lattice automorphism
of the lattice of torsion-free classes. \end{lemma}

\begin{proof} It is clear from the definition that $\phi$ is invertible.  
The fact that $\phi$ is a poset automorphism 
follows from the equivalence of (3) and (5) in 
Lemma~\ref{moreeq}, together with the fact that taking the torsion-free class 
associated to a torsion class reverses containment.  For a finite 
lattice, being a lattice automorphism is equivalent to being a 
poset automorphism, because poset relations determine lattice operations
and vice versa.  
\end{proof}

\begin{lemma}\label{whensplitting} For $w$ a $c$-sortable element, $\mathcal F_w$ is splitting
iff $w$ admits an expression corresponding to an initial segment of the
AR-quiver of $kQ$-mod.
\end{lemma}

\begin{proof} $\mathcal F_w$ is splitting implies that the AR-quiver of
$\mathcal F_w$ is an initial subquiver of the AR-quiver of $\mod kQ$.  
By \cite{AIRT}, we can read off a word for $w$ from the AR-quiver of
$\mathcal F_w$, so this shows that $w$ admits an expression corresponding
to an initial segment of the AR-quiver of $\mod kQ$.  

Conversely, suppose $w$ corresponds to an initial subquiver of
the AR-quiver with respect to an arbitrary linear extension.  Reading this
word by slices gives the $c$-sorting word for 
$w$.  (This uses the fact, shown in \cite{Arm}, that if we think of the
$c$-sorting word for $w_0$ contained in ${\underline c}^\infty$, the $c$-sorting word for
any $w$ will be be contained in the $c$-sorting word for $w_0$.)  By 
\cite{AIRT}, it follows that the AR-quiver for the torsion-free class
corresponding to $w$ coincides with the given initial subquiver of the 
AR-quiver.   It follows that the objects of $\mathcal F_w$ consist of
an initial segment of the AR-quiver of $kQ$-mod.  
\end{proof}

\begin{lemma} \label{fixsplit}
The map $\phi$ is the identity map on the lattice of torsion-free classes. \end{lemma}

\begin{proof} 
We first show that $\phi$ fixes splitting torsion-free classes.  
If 
$\mathcal F_w$ is splitting, then it follows from Lemma~\ref{whensplitting}
that $w$ 
can be read off from an initial segment of the AR-quiver for $\mod kQ$.  
\cite[Proposition 2.8]{PS} establishes that if $w$ comes from an initial 
segment of the AR-quiver for $\mod kQ$, then $w$ is the unique element
of $W$ whose $c$-sortable prefix is $w$.  Thus $w=\hat w$.  

%It follows that $w=\anti_c(w)$, since any element $ws_i$ with $\ell(ws_i)>
%\ell(w)$ has a $c$-sorting word which includes the $c$-sorting word
%for $w$ by the greedy construction of $c$-sorting words.  

The leftmost word for $w$ inside the word for $w_0$ is the word read off
from the initial segment of the AR-quiver, since we know that this is
the $c$-sorting word.  (This is not a complete triviality, 
because an initial segment of the AR-quiver is not typically an initial
segment of our fixed linear order on the indecomposables.)  It now follows
from Theorem~\ref{two.two} that $\mathcal{T}_w$ consist of all sums of
indecomposables not in this initial segment. This is precisely the splitting
torsion class corresponding to $\mathcal{F}_w$.
It follows that $\phi(\mathcal{F}_w) = \mathcal{F}_w$ whenever
$\mathcal{F}_w$ is splitting.

%\comment{Why does this induce the Nakayama automorphism.}
%This induces the Nakayama
%automorphism on $Q$.  

%%The AR-quiver of $kQ^\op$ is, of course, the 
%%reverse of the AR-quiver of $kQ$.  
%%Consider the map $\psi:s_i \to w_0s_iw_0$.  
%%It follows from \cite[Remark 6.5]{cls} that 
%%we can obtain the 
%%$c^{-1}$-sorting word of $w_0$ by reading the 
%%AR-quiver of $kQ$ from right to left and applying $\psi$ to the labels.  

%%This means that when we consider $ww_0=w_0(w_0ww_0)$, we have that
%%$w_0ww_0$ cancels the reflections of the $c^{-1}$-sorting word $w_0$ 
%%corresponding to the objects in $\mathcal F_w$, and therefore 
%%$ww_0$ corresponds
%%to an intial segment of the AR-quiver for $kQ^\op$.  It follows that 
%%$D\mathscr C_{Q^\op}(I_{ww_0})$ consists of the objects in $\mathcal F_w$,
%%and thus $\phi(\mathcal F_w)=\mathcal F_w$ when $\mathcal F_w$ is 
%%a splitting torsion-free class.  

Say that
a torsion-free class is {\it principal} if it is of the form $\Sub(E)$ for
some indecomposable $kQ$-module $E$.  We will now show that $\phi$ fixes
principal torsion-free classes.  

Observe that principal torsion-free classes
can be described in purely lattice-theoretic terms, as the non-zero
torsion-free classes which cannot be written as the join of two smaller
torsion-free classes.  It follows that $\phi$ takes principal torsion-free 
classes 
to principal torsion-free classes.  

Let $E$ be an indecomposable object.  Let $\mathcal S$
be the splitting torsion-free class consisting of the additive hull of
the objects up to and
including $E$ in our standard linear order on the indecomposable $kQ$-modules.  
Let $\mathcal S'$ be 
the additive hull of the indecomposable objects of $\mathcal S$ other than $E$.  Then $\mathcal S'$
is clearly also a splitting torsion-free class.  As we have already seen,
$\phi$ fixes both $\mathcal S$ and $\mathcal S'$.  Since $\Sub E$ is the
only principal torsion-free class contained in $\mathcal S$ but not in
$\mathcal S'$, we have that $\phi(\Sub(E))=\Sub(E)$.
  
Since any torsion-free class can be written as the join of the 
principal torsion-free classes corresponding to the indecomposable summands
of a cogenerator for the torsion-free class, it follows that $\phi$ fixes
all torsion-free classes, as desired.  
\end{proof}

\begin{proposition} Conjecture \ref{associated} holds if $Q$ is of 
finite type.
\end{proposition}

\begin{proof} Suppose that $\mathscr C(I_v)$ is a torsion class.  
By Proposition \ref{conjone}, we know that $vw_0$ is $c^{-1}$-sortable.  
Let $w=\sort_c(v)$.  Applying the above analysis, we find that
$\mathcal F_w=\phi(\mathcal F_w)$, so $\mathcal F_w$ is the torsion-free 
class associated to $\mathscr C(I_v)$, as desired.
\end{proof}

\section{Leftmost reduced words and $\protect\le$-diagrams}

In this section, we explain how our results applied in type $A_n$ 
provide an alternative derivation for Postnikov's description of 
leftmost reduced subwords inside Grassmannian permutations in terms of 
$\le$-diagrams.

Let $W$ be the Weyl group of type $A_n$, which is isomorphic to
the symmetric group on $n+1$ letters.  Fix an integer $k$ such that
$1\leq k \leq n$. 
Let $W_{\langle k\rangle}$ be the 
parabolic subgroup generated by the simple reflections other than $s_k$,
and let $W^{\langle k \rangle}$ be the minimal length coset
representatives for $W_{\langle k\rangle}\backslash W$.  These are the 
$k$-Grassmannian
permutations in $S_{n+1}$ 
(or, depending on a choice of convention, their inverses).
The elements of 
$W^{\langle k \rangle}$ have an essentially unique expression as a product
of simple reflections; if $w\in W^{\ak}$, then any two reduced expressions
for $w$ differ by commutation of commuting reflections.  

Leftmost reduced subwords inside a reduced word for $w\in W^{\ak}$ 
are of interest, 
as they
index the cells in the totally non-negative 
part of the Grassmannian of $k$-planes
in $\mathbb C^{n+1}$.  
In this context, such subwords are referred to as ``positive distinguished
subexpressions'' of $w$.  For more background on this, and for 
the equivalence of ``leftmost reduced'' and ``positive
distinguished,'' see \cite[Section 19]{Postnikov}.

Postnikov gives a combinatorial criterion to identify the lexicographically
first subexpressions in a reduced word for $w \in W^{\ak}$, as follows.  

Let $w_0^\ak$ be the longest element of $W^\ak$.  
A reduced expression for $w_0^\ak$ can be written out explicitly as
$(s_{k}s_{k+1}\dots s_n)(s_{k-1}\dots s_{n-1})\dots (s_1\dots s_{n-k+1})$.  
%There are many other reduced expressions for $w_0^\ak$, but they differ
%only by commutation of commuting reflections: no braid moves are 
%ever applicable to any reduced word for $w_0^\ak$.  
Write out this 
reduced expression  
inside a $k\times (n-k+1)$ rectangle, as is done
in the example below with $n=4,k=2$.  

$$\begin{array}{|c|c|c|} \hline s_2&s_3&s_4 \\  \hline s_1&s_2&s_3 \\ \hline
\end{array}$$

The elements $w\in W^\ak$ correspond bijectively to partitions that
can be drawn inside this rectangle in the usual English notation 
(that
is to say, the parts of the partition are left-justified rows of boxes,
with the sizes of the parts weakly decreasing from top to bottom).  
If $\lambda$ is a partition, a word for the corresponding element of $W^\ak$ is
given by reading the reflections inside $\lambda$ from left to right by
rows, starting at the top row.  
We say that a subword (thought of as a subset of the boxes of this
partition) has a bad $\le$ if there is some reflection which
is used, such that there is some reflection in the column above it which
is unused, and some reflection in the row to the left of it which is 
unused.  (The relative position of the three reflections explains
the use of the symbol $\le$.)

Then Postnikov shows:

\begin{theorem}[{\cite[Theorem 19.2]{Postnikov}}, see also \cite{LW}] 
A subword of $w\in W^{\ak}$ is a leftmost reduced subword iff
it has no bad $\le$.  
\end{theorem}

We will recover this result using our description of leftmost reduced
words in terms of quotient-closed subcategories.  

Let $Q$ be the quiver of type $A_n$, with all arrows oriented away from
vertex $k$.  When we consider the AR quiver for $kQ$-mod, we observe
that it consists of a rectangle $R$ 
and two triangles, $T_1$, $T_2$, as in the picture below, showing 
the case $n=4$, $k=2$.  

%$$\input arquiver.pdf_t $$

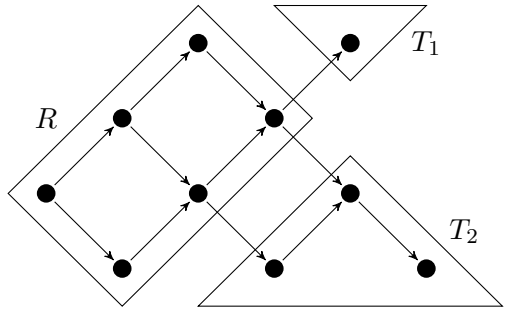
\[ \begin{tikzpicture}
 \node (0) at (0,1) [inner sep=0pt, circle,  fill=black, outer sep=1pt] {$\bullet$};
 \node (1) at (1,2) [inner sep=0pt, circle,  fill=black, outer sep=1pt] {$\bullet$};
 \node (2) at (2,3) [inner sep=0pt, circle,  fill=black, outer sep=1pt] {$\bullet$};
 \node (3) at (1,0) [inner sep=0pt, circle,  fill=black, outer sep=1pt] {$\bullet$};
 \node (4) at (2,1) [inner sep=0pt, circle,  fill=black, outer sep=1pt] {$\bullet$};
 \node (5) at (3,2) [inner sep=0pt, circle,  fill=black, outer sep=1pt] {$\bullet$};
 \node (6) at (4,3) [inner sep=0pt, circle,  fill=black, outer sep=1pt] {$\bullet$};
 \node (7) at (3,0) [inner sep=0pt, circle,  fill=black, outer sep=1pt] {$\bullet$};
 \node (8) at (4,1) [inner sep=0pt, circle,  fill=black, outer sep=1pt] {$\bullet$};
 \node (9) at (5,0) [inner sep=0pt, circle,  fill=black, outer sep=1pt] {$\bullet$};
 \draw [->] (0) -- (1);
 \draw [->] (1) -- (2);
 \draw [->] (0) -- (3);
 \draw [->] (1) -- (4);
 \draw [->] (2) -- (5);
 \draw [->] (3) -- (4);
 \draw [->] (4) -- (5);
 \draw [->] (5) -- (6);
 \draw [->] (4) -- (7);
 \draw [->] (5) -- (8);
 \draw [->] (7) -- (8);
 \draw [->] (8) -- (9);
 \draw (-.5,1) -- (2,3.5) -- (3.5,2) -- (1,-.5) -- cycle; \node at (0,2) {$R$};
 \draw (3,3.5) -- (4,2.5) -- (5,3.5) -- cycle; \node at (5,3) {$T_1$};
 \draw (2,-.5) -- (4,1.5) -- (6,-.5) -- cycle; \node at (5.5,.5) {$T_2$};
\end{tikzpicture} \]

The rectangle $R$ consists of the representations whose support includes
the vertex $k$.  
The lefthand corner
of the rectangle is the simple projective supported at vertex $k$, while the 
righthand corner is the corresponding injective, the unique sincere 
indecomposable representation.  
The triangle $T_1$ consists of 
representations supported only on vertices smaller than $k$, while $T_2$
consists of representations supported only on vertices greater than $k$.

If we replace the indecomposables in the AR-quiver 
by the corresponding simple reflections, and then read them in the order 
given by the 
slices, then Theorem \ref{two.two} tells us that we obtain a word for $w_0$, the 
longest element of $W$.  
Call
this our standard word for $w_0$.  
Say that a reading order \emph{respects the AR-quiver} if, for any
irreducible morphism $A \to B$, we read the reflection corresponding
to $A$ before the reflection corresponding to $B$.  Then any reading
order which respects the AR-quiver, will yield a reduced word for 
$w_0$ which differs from the standard word by a sequence of commutations 
of commuting reflections.  It follows that the leftmost reduced words
for any reading order which respects the AR-quiver will correspond to 
quotient-closed subcategories in exactly the same way.  

In particular, for any $w\in W^{\ak}$, 
we can take a reading order which begins by reading the 
reflections in the corresponding partition 
along lines sloping from bottom left to top right, followed by reading
the remaining reflections in $R$ and the reflections in the 
two triangles in any order 
compatible with the AR-quiver.  The result is a word for $w_0$ which 
begins with a word for $w$.  

By Theorem \ref{two.two}, leftmost reduced words inside $w$ therefore correspond to 
quotient-closed subcategories of $kQ$-mod which contain all the indecomposables
outside the partition corresponding to $w$.  
We have therefore reduced the combinatorial problem
of classifying leftmost reduced words inside $w$ to the problem
of classifying quotient-closed subcategories which contain all the
indecomposables outside a partition $\lambda$.  

We say that a subcategory has a 
bad $\le$ if there is some indecomposable $X$ in $R$ which is missing from the subcategory,
such that the subcategory contains 
an object on the line of morphisms leading to $X$ from the
top right, and an object on the line of morphisms leading to $X$ from the
bottom left.  We will therefore recover Postnikov's result once we have
established the following proposition:

\begin{proposition} The quotient-closed subcategories of $kQ$-mod which 
contain all the indecomposables outside $\lambda$ are exactly those 
which have no bad $\le$ inside $\lambda$.  \end{proposition}

\begin{proof} Suppose that a subcategory $\mathcal C$ has a bad $\le$.  
Then $\mathcal C$ is missing some indecomposable $X$, and contains
an indecomposable $Y$ on the line of morphisms leading to $X$ from the
top right, and an indecomposable $Z$ on the line of morphisms leading
to $X$ from the bottom right.  It is easy to see that there is an
epimorphism $Y\oplus Z \to X$.  Therefore $\mathcal C$ is not
quotient closed. 

Conversely, suppose that $\mathcal C$ contains all the indecomposables
outside $\lambda$ and is not quotient closed.  Then there is some 
indecomposable $X$ which is not in $\mathcal C$, and such that there is
an epimorphism from some object of $\mathcal C$ onto $X$.  It is easy
to see that this is only possible if $\mathcal C$ has a bad $\le$ 
with $X$ at the corner, since all the irreducible morphisms inside 
$R$ are monomorphisms.  
\end{proof}

For this choice of $Q$, 
it is possible to use the same approach to describe the explicit combinatorics
of the leftmost reduced words inside the word for $w_0$ which is
obtained by replacing the indecomposables in the AR-quiver by the
corresponding simple reflections, and then reading them in any order
compatible with the AR-quiver.  

Specifically, we have the following representation-theoretic result:

\begin{proposition} A subcategory $\mathcal C$ of $kQ$ is quotient-closed provided
that: \begin{itemize}
\item $\mathcal C$ has no bad $\le$ inside $R$. 
\item If any indecomposable from $R$ appears in $\mathcal C$, then
so do all the elements of $T_1$ on the same diagonal running from bottom left
to top right, and so do all the elements of $T_2$ on the same diagonal running
from top left to bottom right.  
\item Along any line of morphisms running from bottom left to top right,
if any indecomposable from $T_1$ is in
$\mathcal C$, all subsequent
indecomposables along the diagonal also lie in $\mathcal C$.   
\item Along any line of morphisms running from top left to bottom right,
if any indecomposable from $T_2$ is in
$\mathcal C$, all subsequent
indecomposables along the diagonal also lie in $\mathcal C$.   
\end{itemize}  
\end{proposition}

\begin{proof} We leave the proof of these elementary facts to the reader.
\end{proof}

By Theorem \ref{two.two}, this yields the following consequence.  We think
of the simple reflections in our word for $w_0$ as positioned at the
vertices of the AR-quiver.  In particular, this means that where 
one usually refers to rows and columns, we will refer to diagonals.

\begin{corollary} A leftmost reduced word inside $w_0$ is one which has
the following properties:
\begin{itemize}
\item It has no bad $\le$ inside the reflections coming from $R$,
\item If any simple reflection $s$ inside $R$ is skipped, then all 
subsequent reflections in $T_1$ and $T_2$ 
on the diagonals through that $s$ must also be skipped. 
\item If any simple reflection $s$ inside $T_1$ is skipped, then
all subsequent reflections inside $T_1$ on the same upward-pointing
diagonal must be skipped,
\item If any simple reflection $s$ inside $T_2$ is skipped, then
all subsequent reflections inside $T_2$ on the same downward-pointing
diagonal must be skipped.    
\end{itemize}
\end{corollary}

\section{Connection to the work of Armstrong}

In this section, we explain the link to Armstrong's work \cite{Arm}, which
provided the initial motivation for our investigations.    
We restrict to the case that $Q$ is Dynkin for simplicity; on the
whole, that is the setting in which combinatorialists have worked.  

Let $E$ be a finite ground set and let $\mathcal A$ be a collection
of subsets of $E$.  The sets in $\mathcal A$ are referred to as 
{\it feasible} sets.  
We say that the set system 
$\mathcal A$ is {\it accessible} if,
for every $\emptyset \ne A\in \mathcal A$, there exists some $x\in A$ such that
$A\setminus\{x\}\in \mathcal A$.  

An accessible set system $\mathcal A$ is called an {\it antimatroid} if it
satisfies the condition that if $A,B\in \mathcal A$ with $B\not\subseteq A$,
then there exists some $x\in B\setminus A$ such that $A\cup \{x\}\in \mathcal A$.  

An antimatroid is called {\it supersolvable} \cite{Arm}, if $E$ is 
equipped with a total order such that, if $A,B$ in $\mathcal A$, with
$B\not\subseteq A$, then $A \cup \{x\}\in \mathcal A$, where $x$ is
the minimum element of $B$ not in $A$ (with respect to the total order).

Let $W$ be a Coxeter group, which we assume to be finite.  
Fix an arbitrary word ${{\underline w}=s_{i_1}\dots s_{i_N}}$ in
the simple reflections of $W$.  For $v\in W$, consider the subwords of 
${\underline w}$ which define reduced words for $v$, and, if there is at least
one such subword, define $A_v$ to be the subset of $\{1,\dots,N\}$ 
corresponding to the positions occupied by the leftmost such word.  Then define
$\mathcal A_{\underline w}$ to consist of the collection of all the $A_v$ (for those 
$v$ such that $A_v$ is defined).  
One of the main results of \cite{Arm}, Theorem 4.4, 
says that $\mathcal A_{\underline w}$ is a supersolvable antimatroid (with 
respect to the usual order on the ground set $\{1,\dots,N\}$).

Using our results, we can recover this result of Armstrong 
for particular choices of word ${\underline w}$.   Suppose that $W$ is 
a simply-laced Weyl group, so that it corresponds to a Dynkin diagram.  Choose 
an arbitrary orientation for the diagram, obtaining a quiver $Q$.  
Now consider the word for the element $w_0\in W$ obtained by reading the
 AR-quiver for $kQ$, as described in Section 2.  We call this word
$\underline w_{\rm AR}$.  
%Fix a linear order on the vertices of the AR-quiver of $Q$, which extends the
%order which they naturally posess coming from the quiver structure.  
%Define the
%word ${\underline w}_{AR}$ to be the word for $w_0$ obtained as in 
%obtained by reading the 
%AR-quiver in this fixed order, where an indecomposable object in the 
%$\tau$-orbit of the indecomposable projective $P_i$ is read as $s_i$.  
Using our correspondence between leftmost words and quotient closed subcategories,
$A\in \mathcal A_{{\underline w}_{\rm AR}}$ iff $A$ is the set of indecomposables missing from
a quotient closed subcategory of $kQ$-mod.  In this case we write
$A^c$ for the corresponding quotient closed subcategory.  

Let $A,B \in \mathcal A_{{\underline w}_{\rm AR}}$, with $B\not\subseteq A$.  Let $x$ be the 
first indecomposable (with respect to our fixed total order) in $B$ which
is not in $A$.  To show that $\mathcal A_{{\underline w}_{\rm AR}}$ is a supersolvable antimatroid,
we must show that $A\cup \{x\}\in \mathcal A_{{\underline w}_{\rm AR}}$.  This is
equivalent to saying that $x$ can be removed from the quotient closed
category $A^c$, without destroying quotient closedness.  

To see this, write $X$ for the full subcategory of $\mod kQ$ whose indecomposable objects
properly precede $x$ in our fixed total order.  
Since $x\in B$, we know that $x$ is not a quotient of any object of 
$B^c$. Since $x$ is the first element of $B$ not in $A$, the indecomposable
objects of $B$ which precede $x$ all lie in $A$.  Thus 
$B^c \cap X \supseteq A^c \cap X$, so $A^c\setminus \{x\}$ is still 
quotient closed, as desired.  (As noted in \cite{Arm}, we do not have
to check the fact that $\mathcal A_{{\underline w}_{AR}}$ is accessible separately, since it 
follows from the condition we have already checked.)

\end{document}